\documentclass[a4paper]{article}

\usepackage[pages=all, color=black, position={current page.south}, placement=bottom, scale=1, opacity=1, vshift=5mm]{background}
\SetBgContents{
}    

\usepackage[margin=1in]{geometry} 

\usepackage{amsmath}
\usepackage{amsthm}
\usepackage{amssymb}
\usepackage{enumitem}
\usepackage{tikz-cd} 
\usepackage[utf8]{inputenc}
\usepackage{hyperref}
\hypersetup{
	unicode,
	pdfauthor={Author One, Author Two, Author Three},
	pdftitle={A simple article template},
	pdfsubject={A simple article template},
	pdfkeywords={article, template, simple},
	pdfproducer={LaTeX},
	pdfcreator={pdflatex}
}
\theoremstyle{plain}

\theoremstyle{definition}
\newtheorem{definition}{Definition}[section]

\newtheorem{remark}[definition]{Remark}

\newtheorem{proposition}[definition]{Proposition}
\newtheorem{lemma}[definition]{Lemma}
\newtheorem{theorem}[definition]{Theorem}
\newtheorem{corollary}[definition]{Corollary}

\newtheorem{conjecture}[definition]{Conjecture}

\newtheorem*{teoremaA}{Theorem A}
\newtheorem*{teoremaB}{Theorem B}
\newtheorem*{teoremaC}{Theorem C}
\usepackage{graphicx, color}
\graphicspath{{fig/}}

\usepackage{algorithm, algpseudocode} 
\usepackage{mathrsfs} 

\title{THE WAHL MAP OF THE NORMALIZATION OF NODAL CURVES ON HIRZEBRUCH SURFACES}
\author{Miguel Guerrero-Castillo}
\date{}

\begin{document}
	\maketitle
	
	\begin{abstract}
    In this paper we study the Wahl map for the normalization of a $\delta$-nodal curve $C$ on a Hirzebruch surface $\mathbb{F}_{n}$ for $n\geq 0$. Let $\sigma:X\rightarrow \mathbb{F}_{n}$ be the blow up of $\mathbb{F}_{n}$ along the $\delta$ nodes of $C$ and let $\widetilde{C}$ be the normalization of $C$ under $\sigma$. Let $K_{X}$ be the canonical bundle of $X$ and let $\Omega^{1}_{X}$ be the sheaf of $1$-holomorphic forms on $X$. We give conditions for the surjectivity of the map $\Phi_{X,\mathcal{O}_{X}(K_{X}+\widetilde{C})}: \bigwedge^{2}H^{0}(X,\mathcal{O}_{X}(K_{X}+\widetilde{C}))\rightarrow H^{0}(X,\Omega^{1}_{X}(2K_{X}+2\widetilde{C}))$. Using this surjectivity, we analyze the Wahl map $\Phi_{\widetilde{C}}:\bigwedge^{2}H^{0}(\widetilde{C},\Omega^{1}_{\widetilde{C}})\rightarrow H^{0}(\widetilde{C},(\Omega^{1}_{\widetilde{C}})^{\otimes 3})$ and compute the corank of $\Phi_{\widetilde{C}}$ in various cases. We prove that the corank of the Wahl map for the normalization of a $\delta$-nodal curve on $\mathbb{F}_{n}$ is $h^{0}(\mathbb{F}_{n},\mathcal{O}_{\mathbb{F}_{n}}(-K_{\mathbb{F}_{n}}))$, that verifies a conjecture by Wahl. Furthermore, as an application of our results, we demonstrate that, under certain conditions, a $\delta$-nodal curve on a Hirzebruch surface $\mathbb{F}_{n}$ cannot be embedded as $\delta-$nodal curve on a different Hirzebruch surface $\mathbb{F}_{m}$, for $n\neq m$.\\

	\end{abstract}

	\tableofcontents
	
	\section{Introduction}
	\label{sec:intro}
    Let $C$ be a smooth projective curve of genus $g\geq 2$ and let $\Omega^{1}_{C}$ be the canonical bundle of $C$. The \textit{Wahl map} of $C$ is defined as 
    \[
    \Phi_{C}: \bigwedge^{2}H^{0}(C,\Omega^{1}_{C})\rightarrow H^{0}(C, (\Omega^{1}_{C})^{\otimes 3})
    \]
    and is given by $\tau_{1}\,\wedge\,\tau_{2}\,\mapsto\, \tau_{1}d\tau_{2}-\tau_{2}d\tau_{1}$.\\
    
    A fundamental result, established by J. Wahl \cite{MR1064866}, states that if $C$ is a hyperplane section of a K3 surface, then $\Phi_{C}$ is not surjective. On the other hand, Ciliberto, Harris and Miranda \cite{MR975124} showed that for a general curve (in the sense moduli) of genus $g$, the Wahl map is surjective when $g=10$ or $g\geq 12$. Consequently, a general curve of genus $g\geq 10$ (with $g\neq 11$) cannot lie on a K3 surface  \cite{MR726433,MR977768}. For curves of genus $g\leq 9$, the domain of $\Phi_{C}$ has dimension $\frac{g(g-1)}{2}$, which is strictly less than $5g-5$ (the dimension of the codomain), so $\Phi_{C}$ cannot be surjective. The behavior of the Wahl map for a specific genus has been extensively studied: it is injective for a general curve of genus $g\leq 8$, has a one-dimensional kernel for genus $9$, and a one-dimensional cokernel for genus $11$ \cite{MR1074304}. Cukierman and Ulmer \cite{MR1248892} demonstrated that the corank of the Wahl map is four-dimensional for general curve of genus $10$ which can be embedded on a K3 surface, and the closure of the set of these curves is equal to the locus of curves of genus $10$ for which $\Phi_{C}$ fails to be surjective.\\
    
    In recent years Arbarello, Bruno, and Sernesi \cite{MR3710056} proved that if $C$ is a Brill-Noether-Petri curve of genus $g\geq 12$ with non-surjective Wahl map, then $C$ lies in a polarised K3 surfaces or on a limit thereof.\\
    
    Another significant result concerning the cokernel of the Wahl map is as follows: Let $C\subset \mathbb{P}^{g-1}$ be the canonical curve. The curve $C$ is said to be \textit{extendable} if it is a hyperplane section of a surface $\overline{S}\subset \mathbb{P}^{g}$ that is not a cone. In this context Wahl finds a connection between the cokernel of $\Phi_{C}$ and the deformation theory of the affine cone over $C$ and shows that under some conditions on the obstruction module of the coordinate ring of the affine cone over $C$, the curve $C$ is extendable if and only if $\Phi_{C}$ is not surjective.\\ 

    The Wahl map has also been studied for curves on various projective surfaces. Andreas Leopoldo Knutsen and Angelo Lopez showed under certain conditions the surjectivity of the Wahl map for smooth curves on Enrique surfaces. E. Colombo, P. Frediani, and G. Pareschi showed the surjectivity of the Wahl map (first Gaussian map for the authors) for curves on abelian surfaces under some conditions on multiplication maps of sections of line bundles on the abelian surface. For smooth curves on the projective plane, K3 surfaces, and Hirzebruch surfaces \cite{MR1061775}, the Wahl map has been studied for several authors. In addition, a relevant topic is the case of nodal curve on these surfaces, in this direction \cite{MR1760876} studied the Wahl map for the normalization of plane nodal curves, and the authors showed the surjectivity of the Wahl map for the normalization of a plane nodal curve where the number of nodes is bounded in terms of the degree of the plane curve. On a general primitively polarised K3 surfaces of genus $g+1$, E. Sernesi (\cite{MR3832409}) shows the surjectivity of the Wahl map for the normalization of a $1$-nodal curve provided $g=40,42$ or $\geq 44$.\\
    
    In this work, we follow the ideas developed by Ciliberto, Lopez, Miranda and Sernesi to study the corank of the Wahl map for the normalization of a $\delta$-nodal curve on a Hirzebruch surface.\\
    
    For an integer $n\geq 0$ we consider the Hirzebruch surface $\mathbb{F}_{n}:=\mathbb{P}_{\mathbb{P}^{1}}(\mathcal{O}_{\mathbb{P}^{1}}\oplus \mathcal{O}_{\mathbb{P}^{1}}(-n))$ with canonical fibration map $\phi:\mathbb{F}_{n}\rightarrow\mathbb{P}^{1}$. It is well known that the Picard group of $\mathbb{F}_{n}$ is generated by two class: $F$, which represents a fiber of $\phi$, and $C_{0}$, which is the class of the special section of the fibration, that is, Pic$(\mathbb{F}_{n})=\mathbb{Z}[C_{0}]\oplus \mathbb{Z}[F]$ where the intersections satisfy the relations $C_{0}^{2}=-n$, $C_{0}\cdot F=1$ and $F^{2}=0$. Since Hirzebruch surfaces are ruled surfaces, all fibers are isomorphic and numerically equivalent. The canonical divisor class is $K_{\mathbb{F}_{n}}= -2C_{0}-(2+n)F$.\\
    Let $C\subset \mathbb{F}_{n}$ be a curve with $\delta$ nodal points $p_{1},\cdots, p_{\delta}$ and no other singularities, with $C$ linearly equivalent to $aC_{0}+bF$ for integers $a,b\geq 0$. Let $\sigma: X=\textnormal{Bl}_{Z}\mathbb{F}_{n}\rightarrow\mathbb{F}_{n}$ be the blow-up of $\mathbb{F}_{n}$ at $Z=\{p_{1},\cdots,p_{\delta}\}$, and let $\widetilde{C}$ be the normalization of $C$ under $\sigma$.\\
    
    Our main contributions include establishing conditions for the surjectivity of $\Phi_{X, \mathcal{O}_{X}(K_{X}+\widetilde{C})}$, computing the corank of $\Phi_{\widetilde{C}}$ under various numerical conditions on $a$ and $b$, and verifying a conjecture of Wahl (Conjecture \ref{conjeturawahl}) in our context, that is, the corank of the Wahl map for the normalization of a $\delta$-nodal curve is $h^{0}(\mathbb{F}_{n},\mathcal{O}_{\mathbb{F}_{n}}(-K_{\mathbb{F}_{n}}))$ and this value is greater than or equal to $h^{0}(X,\mathcal{O}_{X}(-K_{X}))$. Finally, we prove that, under certain conditions, a $\delta-$nodal curve on $\mathbb{F}_{n}$ cannot lie on a different Hirzebruch surfaces $\mathbb{F}_{m}$, as $\delta-$nodal curve, for $n\neq m$. Our results contribute to understanding certain loci of curves with non-surjectivity Wahl map in  the moduli spaces of smooth curves of genus $g$. This topic is a work in progress on the matter.\\

The paper is organized as follows. Section \ref{sec:pre} presents the preliminaries and establishes our notation. Section \ref{sec:3} contains our main result on the Wahl map for the normalization of nodal curves. Section \ref{sec:surjectivity} studies the surjectivity of Gaussian maps on the blow-up surface; together with the results of Section \ref{sec:3}, this yields the following theorem:
    \begin{teoremaA}(Theorem \ref{teoremafinaldeltanodos}).
    Let $p_{1},\dots,p_{\delta}$ be $\delta$ distinct points on $\mathbb{F}_{n}$ such that for each $j$, $p_{j}\notin C_{0}$, and for $i,j\in\{1,\dots,\delta\}$, $p_{i}$ and $p_{j}$ are not on the same fiber of the fibration that admits $\mathbb{F}_{n}$, $\phi:\mathbb{F}_{n}\rightarrow\mathbb{P}^{1}$.\\
Suppose that $C$ is a nodal curve on $\mathbb{F}_{n}$ in the linear system $|D|$, with $D$ linearly equivalent to $aC_{0}+bF$.
Let $\sigma:X=Bl_{Z}\mathbb{F}_{n}\rightarrow\mathbb{F}_{n}$ be the blow-up of $\mathbb{F}_{n}$ at $Z$, where $Z=\{p_{1},\dots,p_{\delta}\}$ are the nodes of $C$. Let $\widetilde{C}$ be the normalization of $C$ under $\sigma$.\\ Assume that $a\geq 6$ and $b\geq \max\{(a+7)n,(a-1)n+\delta+2,6\delta-3n+3\}$,
then:
\begin{enumerate}
    \item  the Gaussian map $\Phi_{X,\mathcal{O}_{X}(K_{X}+\widetilde{C})}$ 
    is surjective.
    \item $\textnormal{corank}(\Phi_{\widetilde{C}})=h^{0}(\mathbb{F}_{n},\mathcal{O}_{\mathbb{F}_{n}}(-K_{\mathbb{F}_{n}})) = \left\{ \begin{array}{lcc} 9 & & n\leq 2,  \\ &  & \\  n+6 &  & n\geq 3. \end{array} \right.$.
\end{enumerate}
    \end{teoremaA}
In Section \ref{sec:applications and special cases} discusses applications and special cases, particularly for $1$-nodal curves,  and we demonstrate the following main results:
\begin{teoremaB}(Theorem \ref{teoremafinal1nodo})
        Suppose that $C$ is a $1-$nodal curve on $\mathbb{F}_{n}$, linearly equivalent to $aC_{0}+bF$ with $a,b\geq 0$. Let $\sigma:X=Bl_{p}(\mathbb{F}_{n})\rightarrow\mathbb{F}_{n}$ be the blow-up of $\mathbb{F}_{n}$ at $p$, where $p\in C$ is the node. Let $\widetilde{C}$ be the normalization of $C$ under $\sigma$. Assume that $p\notin C_{0}$, $a\geq 6$ and $b\geq\max\{(a-2)n+6,an+3\}$, then
    \begin{enumerate}
        \item $\Phi_{X,\mathcal{O}_{X}(K_{X}+\widetilde{C})}$ is surjective.\\
        \item $
   \textnormal{corank} \Phi_{\widetilde{C}} = \left\{ \begin{array}{lcc} 9 & & n\leq 2,  \\ &  & \\  n+6 &  & n\geq 3. \end{array} \right.
$
    \end{enumerate}
\end{teoremaB}
\begin{teoremaC}(Corollary \ref{corolario1nodo})
    Under the conditions of Theorem \ref{teoremafinaldeltanodos}, if there exists a $\delta-$nodal curve $C\in |aC_{0}+bF|$ in $\mathbb{F}_{n}$, then $C$ cannot be embedded in $\mathbb{F}_{m}$ as $\delta-$nodal curve, for $m\neq n$ and $m\geq 4$.
\end{teoremaC}	

\paragraph{Acknowledgements.} 
\textit{This work was supported by a Ph.D. fellowship from SECIHTI–Mexico and forms part of the author’s Ph.D. thesis. The author gratefully acknowledges the Centro de Ciencias Matemáticas UNAM for its institutional support and for granting access to its facilities throughout the doctoral program. The author expresses sincere gratitude to Prof. Abel Castorena for his guidance and continuous support during the development of this research, and to Prof. Angelo Felice Lopez for valuable ideas and insightful discussions during an academic visit.}
	\section{Preliminaries}
	\label{sec:pre}
    We work in the field of complex numbers $\mathbb{C}$. 
    \subsection{Gaussian maps and Wahl map}
    \begin{definition}
    Let $C$ be a smooth projective curve and $L$ a line bundle on $C$. The \textit{Gaussian map} for $L$ on $C$ is the map
    \begin{equation}
    \Phi_{C,L}:\bigwedge^{2} H^{0}(C,L)\rightarrow H^{0}(C,L^{2}\otimes\Omega_{C}^{1})
    \end{equation}
    given in terms of a local coordinate $z$ on $C$ and a trivialization of $L$ by sending $f(z)\wedge g(z)$ to $[fg' - gf']dz$. 
    \end{definition}
    \begin{definition}
     The \textit{Wahl map} of $C$ is the Gaussian map with $L=\Omega^{1}_{C}$ the canonical bundle of $C$:   
     \begin{equation}
    \Phi_{C}:\bigwedge^{2} H^{0}(C,\Omega_{C}^{1})\rightarrow H^{0}(C,(\Omega_{C}^{1})^{\otimes 3}).
    \end{equation}
    \end{definition}
    The definition of Gaussian map can be extended to surfaces. 
    \begin{definition}
        Let $S$ be a smooth surface and $M$ a line bundle on $S$. The \textit{Gaussian map} for $M$ on $S$ is the map
    \begin{equation}
    \Phi_{S,M}:\bigwedge^{2} H^{0}(S,M)\rightarrow H^{0}(S,M^{2}\otimes\Omega_{S}^{1})
    \end{equation}
    defined locally by  $f(z, w)\wedge g(z, w) \mapsto f\cdot dg- g\cdot df$, where ${z, w}$ are local coordinates on $S$.
    \end{definition}

    \subsection{The Standard Diagram}
    Consider a smooth curve  $C$ on a smooth surface $S$. If we take the line bundle $M$ on $S$ to be $\mathcal{O}_{S}(K_{S} + C)$, then by the adjunction formula $M\mid_{C}=\Omega^{1}_{C}$ and we have a commutative diagram
    \begin{equation} \label{DiagramaD}
        \tag{D}
           \begin{tikzcd}
        \bigwedge^2 H^{0}(S, \mathcal{O}_{S}(K_{S}+C)) \arrow{rrr}{\Phi_{S,\mathcal{O}_{S}(K_{S}+C)}} \arrow[swap]{dd}{\textnormal{Res}} & & & H^{0}(S,\Omega^{1}_{S}(2K_{S}+2C)) \arrow{d}{\alpha} \arrow[dd, bend left=85, "H^{0}(\rho)=\beta\, \circ \, \alpha"] \\%
        & & & H^{0}(C, \Omega^{1}_{S}(2K_{S}+2C)\mid_{C})  \arrow{d}{\beta} \\%
        \bigwedge^{2} H^{0}(C, \Omega^{1}_{C}) \arrow[swap]{rrr}{\Phi_{C}} && & H^{0}(C, (\Omega^{1}_{C})^{\otimes 3})
      \end{tikzcd}
    \end{equation}
    where the map Res is defined by $\bigwedge^{2}$res, and $\textnormal{res}:H^{0}(S, \mathcal{O}_{S}(K_{S}+C))\rightarrow H^{0}(C, \Omega^{1}_{C})$ is the restriction to $C$, and the vertical maps $\alpha$ and $\beta$ on the right are defined as follows: we will denote by $H^{0}(\rho)$ the composition $\beta\circ\alpha$.\\
    Consider the classical short exact sequence on $S$:
    \begin{equation} \label{se1}
     0\rightarrow \mathcal{O}_{S}(-C) \rightarrow \mathcal{O}_{S} \rightarrow \mathcal{O}_{C} \rightarrow 0.
    \end{equation}
    Tensoring the sequence (\ref{se1}) with $\mathcal{O}_{S}(2K_{S}+2C)\otimes \Omega^{1}_{S}$ yields:
    \begin{equation} \label{sucesioneuleromega}
        0\rightarrow \Omega^{1}_{S}(2K_{S}+C) \rightarrow \Omega^{1}_{S}(2K_{S}+2C) \rightarrow \Omega^{1}_{S}(2K_{S}+2C)\mid_{C} \rightarrow 0,
    \end{equation}
    and the map $\alpha$ appears in the long exact sequence in cohomology,
    \begin{equation*}
        0 \rightarrow H^{0}(S, \Omega^{1}_{S}(2K_{S}+C)) \rightarrow H^{0}(S,\Omega^{1}_{S}(2K_{S}+2C))\xrightarrow{\alpha}H^{0}(C, \Omega^{1}_{S}(2K_{S}+2C)\mid_{C})\rightarrow \cdots
    \end{equation*}
    For the map $\beta$, we use the dual of the normal sequence for $C$ in $S$
    \begin{equation} \label{dualnormal}
    0\rightarrow N^{\vee}_{C/S}\rightarrow \Omega^{1}_{S}\mid_{C}\rightarrow \Omega^{1}_{C}\rightarrow 0,
    \end{equation}
    where $N_{C/S}=\mathcal{O}_{S}(C)\mid_{C}$. Tensoring the sequence (\ref{dualnormal}) by $\mathcal{O}_{S}(2K_{S}+2C)\mid_{C}=(\Omega^{1}_{C})^{\otimes 2}$, we have the exact sequence:
    \begin{equation}\label{sucesiondualnormal}
    0\rightarrow \mathcal{O}_{S}(2K_{S}+C)\mid_{C} \rightarrow \Omega^{1}_{S}(2K_{S}+2C)\mid_{C} \rightarrow (\Omega^{1}_{C})^{\otimes 3} \rightarrow 0,
     \end{equation}
    which gives the map $\beta$
    \begin{equation*}
        0 \rightarrow H^{0}(C, \mathcal{O}_{S}(2K_{S}+C)\mid_{C})\rightarrow H^{0}(C,\Omega^{1}_{S}(2K_{S}+2C)\mid_{C})\xrightarrow{\beta} H^{0}(C, (\Omega^{1}_{C})^{\otimes 3})\rightarrow \cdots
    \end{equation*}
    \begin{remark} \label{remarkcorank1}
    From diagram $\ref{DiagramaD}$, we have
    \[
    \textnormal{corank}(\Phi_{C}\circ \textnormal{Res}) = \textnormal{corank}(H^{0}(\rho) \circ \Phi_{S, \mathcal{O}_{S}(K_{S}+C)}).
    \]
    In particular, if both Res and $\Phi_{S, \mathcal{O}_{S}(K_{S}+C)}$ are surjective, then $\textnormal{corank}(\Phi_{C})=\textnormal{corank}(H^{0}(\rho))$.
    \end{remark}

    \subsection{Logarithmic Sheaves}
    To study the corank of $H^{0}(\rho)$, we employ sheaves of differential forms with logarithmic poles. Let $\Omega^{1}_{S}(\log C)$ denote the sheaf of $1$-forms with logarithmic poles along $C$. Locally, if $z_{1},z_{2}$ are local coordinates on $S$ with $C = \{z_2 =0\}$, then $\Omega^{1}_{S}(\log C)$ is locally generated by $dz_{1},\frac{dz_{2}}{z_{2}}$.\\

    From (\cite{MR1193913}, 2.3c), we have the exact sequence:
    \begin{equation}\label{sucesionlogarithmic}
    0\rightarrow \Omega^{1}_{S}(\log C)(-C)\rightarrow\Omega^{1}_{S} \rightarrow\Omega^{1}_{C}\rightarrow 0,
    \end{equation}
    note that the kernel of $\Omega^{1}_{S}\rightarrow\Omega^{1}_{C}$ is locally generated by $z_{2}dz_{1}$, $dz_{2}$. But these are exactly the local generators of the subsheaf $\Omega^{1}_{S}(\log C)(-C)\subset\Omega^{1}_{S}(\log C)$, for further details, see \cite{MR1193913, MR2095471}.\\
    Tensoring (\ref{sucesionlogarithmic}) with $\mathcal{O}_{S}(2K_{S}+2C)$ we obtain
    \begin{equation}
    \label{sucesionlogarithmic2}
    0\rightarrow \Omega^{1}_{S}(\log C)(2K+C)\rightarrow\Omega^{1}_{S}(2K_{S}+2C) \xrightarrow{\rho}\left(\Omega^{1}_{C}\right)^{\otimes3}\rightarrow 0
    \end{equation}
    which gives the map $H^{0}(\rho)$
    \begin{equation*}
        0 \rightarrow H^{0}(S, \Omega^{1}_{S}(\log C)(2K+C))\rightarrow H^{0}(S,\Omega^{1}_{S}(2K+2C))\xrightarrow{H^{0}(\rho)} H^{0}(C, (\Omega^{1}_{C})^{\otimes 3})\rightarrow \cdots,
    \end{equation*}
    where $H^{0}(\rho)$ is exactly the composition  $\beta \,\circ\, \alpha$.\\
    Another relevant exact sequence from (\cite{MR1193913}, 2.3a) in our case, by tensoring with $\mathcal{O}_{S}(2K_{S}+C)$, is as follows:
    \begin{equation}\label{sucesionlogarithmic3}
    0 \rightarrow \Omega^{1}_{S}(2K_{S}+C)\rightarrow \Omega^{1}_{S}(\log C)(2K_{S}+C) \rightarrow\mathcal{O}_{S}(2K_{S}+C)\mid_{C}\rightarrow 0.
    \end{equation}
    We will use these exact sequence extensively when $S$ is the blow up of a Hirzebruch surface $\mathbb{F}_{n}$ and $C$ is the normalization of a nodal curve. 

   \subsection{Demostrations of some properties of higher direct images of sheaves on the blow-up surface} 
Let $\sigma:X=\textnormal{Bl}_{Z}\mathbb{F}_{n}\rightarrow\mathbb{F}_{n}$ be the blow up of $\mathbb{F}_{n}$ at $Z=\{p_{1},\cdots,p_{\delta}\}$ where $p_{1},\cdots,p_{\delta}$ are $\delta$ distinct points with the exceptional divisor $E=E_{1}+E_{2}+\cdots+E_{\delta}$ and $E_{j}=\sigma^{-1}(p_{j})$ for $j=1,\cdots,\delta$.
\begin{itemize}
    \item  $R^{i}\sigma_{\ast}\mathcal{O}_{X}(-rE)=0$ for $r\in\mathbb{Z}_{\geq0}$ and $i\geq 1$.
 \begin{proof}
Note that the sheaves $R^{i}\sigma_{*}\mathcal{O}_{X}(-rE)$ for $i>0$ have support and $Z$. We use the theorem on formal function to compute these sheaves at each point. Consider $\widehat{E}:=E_{j}$ and $p:=p_{j}$ for some $j\in\{1,\cdots, \delta\}$.
For each $t\geq 1$, we define $X_{(t)}:=X \times_{\mathbb{F}_{n}}\textnormal{Spec}(\mathcal{O}_{p}/\mathfrak{m}_{p}^{t})$, note that for $t=1$ we get the fiber at $p$, that is $X_{(1)}\cong \widehat{E}$, and for $t>1$ we get a scheme with nilpotent elements having the same underlying space as $X_{(1)}$, that is, $X_{(t)}\cong t\widehat{E}$.\\
     \[
\begin{tikzcd}
 X_{t}\arrow{r}{\nu}\arrow{d} & X \arrow{d}{\sigma}\\
 \textnormal{Spec } \mathcal{O}_{p}/\mathfrak{m}_{p}^{t} \arrow{r} & \mathbb{F}_{n}
 \end{tikzcd}
 \]
 Let $\mathcal{G}_{t}:=\nu^{\ast}\mathcal{O}_{X}(-rE)$, where $\nu:X_{t}\rightarrow X$ is the natural map. The Theorem on Formal Functions (\cite{MR463157},III,11.1) says that 
 \begin{equation*}
     R^{i}\sigma_{*}\mathcal{O}_{X}(-rE)^{\widehat{}}_{p}=\lim_{\longleftarrow}H^{i}(X_{(t)}, \mathcal{G}_{t}).
 \end{equation*} 
 Note that, $\nu^{\ast}\mathcal{O}_{X}(-rE)=\nu^{-1}\mathcal{O}_{X}(-rE)\otimes_{\nu^{-1}\mathcal{O}_{X}}\mathcal{O}_{X_{t}}\cong \mathcal{O}_{X}(-rE)\mid_{X_{t}}$.\\
 We have the exact sequence on $X$
 \begin{equation*}
     0\rightarrow \mathcal{O}_{\widehat{E}}(-t\widehat{E})\rightarrow \mathcal{O}_{t\widehat{E}+\widehat{E}}\rightarrow \mathcal{O}_{t\widehat{E}}\rightarrow 0,
 \end{equation*}
 then
 \begin{equation*}
     0\rightarrow \mathcal{O}_{\widehat{E}}(-t\widehat{E}-rE)\rightarrow \mathcal{G}_{t+1}\rightarrow \mathcal{G}_{t}\rightarrow 0.
 \end{equation*}
 Note that, $H^{i}(\widehat{E}, \mathcal{O}_{\widehat{E}}(-t\widehat{E}-rE))\cong H^{i}(\mathbb{P}^{1}, \mathcal{O}_{\mathbb{P}^{1}}(t+r))$ since $\mathcal{O}_{\widehat{E}}(\widehat{E})\ \cong \mathcal{O}_{X}(-1)\mid_{\widehat{E}}\cong \mathcal{O}_{\mathbb{P}^{1}}(-1)$ and $\mathcal{O}_{\widehat{E}}(-t\widehat{E}-rE)\cong\mathcal{O}_{\widehat{E}}(-t\widehat{E}-r\widehat{E})$.\\
 Thus, for each $i\geq 1$, $H^{i}(\widehat{E}, \mathcal{O}_{\widehat{E}}(-t\widehat{E}-rE))=0$ for all $t\geq 1$. Therefore, for $t\geq 1$
 \begin{equation*}
     H^{i}(X_{(t+1)}, \mathcal{G}_{t+1})\cong H^{i}(X_{(t)}, \mathcal{G}_{t}).
 \end{equation*}
 Since $H^{i}(X_{(1)},\mathcal{G}_{1})= H^{i}(\widehat{E}, \mathcal{O}_{\widehat{E}}(-rE))\cong  H^{i}(\mathbb{P}^{1},\mathcal{O}_{\mathbb{P}^{1}}(r))\cong 0$ then $H^{i}(X_{t}, \mathcal{F}_{t})=0$ for all $r\geq 0$ and $i\geq1$, and
 \begin{equation*}
     R^{i}\sigma_{\ast}\mathcal{O}_{X}(-rE)^{\widehat{}}_{p}=0.
 \end{equation*}
 Since $R^{i}\sigma_{*}\mathcal{O}_{X}(-rE)$ is a coherent sheaf with support at $Z$, and $R^{i}\sigma_{*}\mathcal{O}_{X}(-rE)^{\widehat{}}_{p_{j}}=0$ for all $p_{j}\in Z$, then $R^{i}\sigma_{*}\mathcal{O}_{X}(-rE)=0$.\\
 Further, using the exact sequence
 \begin{equation*}
    0 \rightarrow \mathcal{O}_{X}(-E) \rightarrow \mathcal{O}_{X} \rightarrow \mathcal{O}_{E}\rightarrow 0,
\end{equation*}
we have a long exact sequence 
\begin{align*} 
0 & \rightarrow \sigma_{\ast}\mathcal{O}_{X}(-E) \rightarrow \sigma_{\ast}\mathcal{O}_{X} \rightarrow \sigma_{\ast}\mathcal{O}_{E} \rightarrow \\ 
 & \rightarrow R^{1}\sigma_{\ast}\mathcal{O}_{X}(-E) \rightarrow R^{1}\sigma_{\ast}\mathcal{O}_{X}\rightarrow R^{1}\sigma_{\ast}\mathcal{O}_{E} \rightarrow \\
 & \rightarrow R^{2}\sigma_{\ast}\mathcal{O}_{X}(-E) \rightarrow R^{2}\sigma_{\ast}\mathcal{O}_{X}\rightarrow R^{2}\sigma_{\ast}\mathcal{O}_{E} \rightarrow 0,
\end{align*}
it is well known that $\sigma_{\ast}\mathcal{O}_{X}=\mathcal{O}_{\mathbb{F}_{n}}, \sigma_{\ast}\mathcal{O}_{E}=\mathcal{O}_{Z}$, then we can conclude that  $\sigma_{\ast}\mathcal{O}_{X}(-E)=\mathcal{I}_{Z}$, where $\mathcal{I}_{Z}$ is the ideal sheaf of $Z$ on $\mathbb{F}_{n}$. Similarly, for $r\geq1$, we can show that $\sigma_{*}\mathcal{O}_{X}(-rE)\cong \mathcal{I}_{Z}^{r}$ . \\
 \end{proof}
  
	\item $R^{1}\sigma_{*}\Omega^{1}_{X}(\log\widetilde{C})=0$:
    \begin{proof}
    Set $\widehat{E}:=E_{j}$ and $p:=p_{j}$ for some $j\in \{1,\dots,\delta\}$.\\
    For each $t\geq1$, we define $X_{(t)}:=X\times_{\mathbb{F}_{n}}Spec (\mathcal{O}_{\mathbb{F}_{n},p}/\mathfrak{m}^{t}_{p})$. Note that for $t=1$, we get the fiber at $p$ that is $X_{(1)}\cong \widehat{E}$, and for $t>1$ we get a scheme with nilpotent elements having the same underlying space as $X_{(1)}$, that is, $X_{(t)}\cong t\widehat{E}$. We have the diagram
      \[
\begin{tikzcd}
 X_{(t)}\arrow{r}{\nu}\arrow{d} & X \arrow{d}{\sigma}\\
 \textnormal{Spec } \mathcal{O}_{p}/\mathfrak{m}_{p}^{t} \arrow{r} & \mathbb{F}_{n}
 \end{tikzcd}
 \]
 Let $\mathcal{F}_{t}:=\nu^{*}\Omega^{1}_{X}(\log \widetilde{C})$, where $\nu: X_{(t)}\rightarrow X$ is the natural map. The Theorem on Formal Functions (\cite{MR463157},III,11.1) says that 
 \begin{equation*}
     R^{1}\sigma_{*}\Omega^{1}_{X}(\log\widetilde{C})^{\widehat{}}_{p}=\lim_{\longleftarrow}H^{1}(X_{(t)}, \mathcal{F}_{t}).
 \end{equation*}
 Note that, $\nu^{*}\Omega^{1}_{X}(\log \widetilde{C})=\nu^{-1}\Omega^{1}_{X}(\log \widetilde{C})\otimes_{\nu^{-1}\mathcal{O}_{X}}\mathcal{O}_{X_{(t)}}\cong \Omega^{1}_{X}(\log \widetilde{C})|_{X_{(t)}}$.\\
 We have the exact sequence on $X$
 \begin{equation*}
     0 \rightarrow \mathcal{O}_{\widehat{E}}(-t\widehat{E})\rightarrow \mathcal{O}_{t\widetilde{E}+\widetilde{E}}\rightarrow \mathcal{O}_{t\widehat{E}}\rightarrow 0,
 \end{equation*}
 then
 \begin{equation*}
     0 \rightarrow \Omega^{1}_{X}(\log \widetilde{C})\otimes\mathcal{O}_{\widehat{E}}(-t\widehat{E})\rightarrow \Omega^{1}_{X}(\log \widetilde{C})\otimes\mathcal{O}_{X_{(t+1)}}\rightarrow \Omega^{1}_{X}(\log \widetilde{C})\otimes\mathcal{O}_{X_{(t)}}\rightarrow 0,
 \end{equation*}
 that is
 \begin{equation*}
     0 \rightarrow \mathcal{F}_{1}\otimes\mathcal{O}_{\widehat{E}}(-t\widehat{E})\rightarrow \mathcal{F}_{t+1 }\rightarrow \mathcal{F}_{t }\rightarrow 0.
 \end{equation*}
 Now, we will show that $H^{1}(\mathcal{F}_{1}\otimes\mathcal{O}_{\widehat{E}}(-t\widehat{E}))=0$, which implies that $H^{1}(X_{(t+1)},\mathcal{F}_{t+1})\cong H^{1}(X_{(t)},\mathcal{F}_{t})$, that is $R^{1}\sigma_{*}\Omega^{1}_{X}(\log\widetilde{C})^{\widehat{}}_{p}$ is trivial. Since $R^{1}\sigma_{*}\Omega^{1}_{X}(\log \widetilde{C})$ is a coherent sheaf with support at $Z$, and $R^{1}\sigma_{*}\Omega^{1}_{X}(\log\widetilde{C})^{\widehat{}}_{p_{j}}=0$ for all $p_{j}\in Z$, then $R^{1}\sigma_{*}\Omega^{1}_{X}(\log\widetilde{C})=0$.\\
    By tensoring the sequence (\ref{secotang2}) by $\mathcal{O}_{\widehat{E}}$, we obtain
\begin{equation*}
    0\rightarrow Tor^{1}_{X}(\mathcal{O}_{\widehat{E}}(2\widehat{E}),\mathcal{O}_{\widehat{E}})\rightarrow \left( \sigma^{*}\Omega^{1}_{\mathbb{F}_{n}} \right)|_{\widehat{E}} \rightarrow \left(\Omega^{1}_{X} \right)|_{\widehat{E}} \rightarrow \mathcal{O}_{\widehat{E}}(2\widehat{E})\rightarrow 0.
\end{equation*}
Observe that $Tor^{1}_{X}(\mathcal{O}_{\widehat{E}}(2\widehat{E}),\mathcal{O}_{\widehat{E}})\cong \mathcal{O}_{\widehat{E}}(\widehat{E})$ and $\left( \sigma^{*}\Omega^{1}_{\mathbb{F}_{n}} \right)|_{\widehat{E}}\cong \mathcal{O}_{\widehat{E}}\oplus \mathcal{O}_{\widehat{E}}$, this implies that $\left(\Omega^{1}_{X} \right)|_{\widehat{E}} \cong \mathcal{O}_{\widehat{E}}(2\widehat{E})\oplus \mathcal{O}_{\widehat{E}}(-\widehat{E})$.\\
Let $B:=\widetilde{C}\cap \widehat{E}=\{q_{1},q_{2}\}$ be the intersection points since $\widetilde{C}.\widehat{E}=2$. Consider the inclusion $i:\widehat{E}\hookrightarrow X$, the diagram of (\cite{MR4665627}, Lemma 2.2) in our setting is as follows
       \[
\begin{tikzcd}
               &  0  \arrow{d} & 0  \arrow{d} &  0  \arrow{d} &    \\
    0 \arrow{r}  &  T_{\widehat{E}}(-\log B) \arrow{r} \arrow{d} & i^{*}T_{X}(-\log\widetilde{C}) \arrow{r} \arrow{d} &  N_{\log \widetilde{C}/B} \arrow{r} \arrow{d} &  0  \\
    0 \arrow{r}  &  T_{\widehat{E}} \arrow{r} \arrow{d} & i^{*}T_{X} \arrow{r} \arrow{d} &  N_{\widehat{E}|X} \arrow{r} \arrow{d} &  0   \\
    0 \arrow{r}  &  \mathcal{O}_{B} \arrow{r} \arrow{d} & i^{*}\mathcal{O}_{\widetilde{C}}(\widetilde{C}) \arrow{r} \arrow{d} &  0 \arrow{r} \arrow{d} &  0  \\
                &  0  & 0  &  0  &  
\end{tikzcd}
\]
where $T_{X}(-\log \widetilde{C})$ is called the logarithmic tangent sheaf associated to $\widetilde{C}$, and is the dual sheaf of $\Omega^{1}_{X}(\log \widetilde{C})$;  it is the sheaf of holomorphic vector fields tangent to $\widetilde{C}$. Similarly, for the sheaf $T_{\widehat{E}}(-\log B)$ on $\widehat{E}$. Since
$T_{\widehat{E}}\cong T_{\mathbb{P}^{1}}\cong\mathcal{O}_{\mathbb{P}^{1}}(2)$ and $\mathcal{O}_{B}=\mathcal{O}_{q_{1}}\oplus\mathcal{O}_{q_{2}}$ then we conclude that $T_{\widehat{E}}(-\log B)\cong \mathcal{O}_{\mathbb{P}^{1}}$. As $N_{\widehat{E}|X}\cong \mathcal{O}_{\widehat{E}}(\widehat{E})\cong\mathcal{O}_{\mathbb{P}^{1}}(-1)$ then $i^{*}T_{X}=T_{X}|_{\widehat{E}}\cong\mathcal{O}_{\mathbb{P}^{1}}(2)\oplus\mathcal{O}_{\mathbb{P}^{1}}(-1)$. Therefore $i^{*}T_{X}(-\log \widetilde{C})=T_{X}(-\log\widetilde{C})|_{\widehat{E}}\cong\mathcal{O}_{\mathbb{P}^{1}}(-1)\oplus\mathcal{O}_{\mathbb{P}^{1}}$ since 
$N_{\log \widetilde{C}/B}\cong \mathcal{O}_{\mathbb{P}^{1}}(-1)$.\\
Thus, we obtain
\begin{equation*}
    \mathcal{F}_{1}=\Omega^{1}_{X}(\log \widetilde{C})|_{\widehat{E}}\cong \mathcal{O}_{\mathbb{P}^{1}}(1)\oplus\mathcal{O}_{\mathbb{P}^{1}}
\end{equation*}
and therefore $\mathcal{F}_{1}\otimes\mathcal{O}_{\widehat{E}}(-t\widehat{E})\cong ( \mathcal{O}_{\mathbb{P}^{1}}(1)\oplus\mathcal{O}_{\mathbb{P}^{1}})\otimes \mathcal{O}_{\mathbb{P}^{1}}(t)\cong \mathcal{O}_{\mathbb{P}^{1}}(t+1)\oplus\mathcal{O}_{\mathbb{P}^{1}}(t)$.
In particular, we have that $H^{1}(\mathcal{F}_{1}\otimes \mathcal{O}_{\widehat{E}}(-t\widehat{E}))\cong 0$.
\end{proof}
\end{itemize}
We will use these properties in the next section.
   
     \section{The Wahl map of the normalization of a $\delta$-nodal curve on $\mathbb{F}_{n}$}
\label{sec:3}
     \subsection{Setup and Notation}
     Let $\mathbb{F}_{n}$ be the Hirzebruch surface Proj$(\mathcal{O}_{\mathbb{P}^{1}}\oplus\mathcal{O}_{\mathbb{P}^{1}}(-n))$, $n\geq0$; and let $\phi:\mathbb{F}_{n}\rightarrow\mathbb{P}^{1}$ the canonical projection map. It is well known that the Picard group of $\mathbb{F}_{n}$ is generated by two class: $F$, which represents a fiber, and $C_{0}$, which is the class of the special section of the fibration of $\phi$. In other words, Pic$(\mathbb{F}_{n})=\mathbb{Z}[C_{0}]\oplus\mathbb{Z}[F]$, where the intersections satisfy the relations $C_{0}^{2}=-n$, $C_{0}\cdot F=1$ and $F^{2}=0$. Since Hirzebruch surfaces are ruled, all fibers are isomorphic and numerically equivalent. The canonical divisor class $K_{\mathbb{F}_{n}}$ is given by $K_{\mathbb{F}_{n}}=-2C_{0}-(2+n)F$.\\
     
     Let $C\subset \mathbb{F}_{n}$ a curve with $\delta$ nodal points $p_{1},\dots,p_{\delta}$ and no other singularities. Let 
     \[
     \sigma: X=Bl_{Z}\mathbb{F}_{n}\rightarrow\mathbb{F}_{n}
     \]
     be the blow-up of $\mathbb{F}_{n}$ at $Z=\{p_{1},\dots,p_{\delta}\}$ with exceptional divisor $E=\sum_{j=1}^{\delta}E_{j}$, where $E_{j}=\sigma^{-1}(p_{j})$. The intersection numbers satisfy $E_{j}^{2}=-1$ and $E_{j}\cdot E_{i}=0$ for $i\neq j$, thus $E^{2}=-\delta$. It is well known that Pic$(X)\cong\sigma^{*}\textnormal{Pic}(\mathbb{F}_{n})\oplus \mathbb{Z}[E_{1}]\oplus \cdots \mathbb{Z}[E_{\delta}]$.\\
     Let $\widetilde{C}$ be the normalization of $C$ under $\sigma$, with $C$ linearly equivalent to $aC_{0}+bF$ for some integers $a,b\geq 0$. Then $\widetilde{C}=\sigma^{*}C-2E$ and the canonical divisor on $X$ is $K_{X}=\sigma^{*}K_{\mathbb{F}_{n}}+E$. We have the intersection numbers: 
     \begin{itemize}
            \item $\widetilde{C}^{2}=C^{2}+4E^{2}=C^{2}-4\delta=2ab-a^{2}n-4\delta$,
            \item $K_{X}^{2}=K_{\mathbb{F}_{n}}^{2}+E^{2}=8-\delta$,
             \item $K_{X}\cdot\widetilde{C}=K_{\mathbb{F}_{n}}\cdot C-2E^{2}=K_{\mathbb{F}_{n}}\cdot C+2\delta=an-2a-2b+2\delta$.
     \end{itemize} 
     The arithmetic genus of $C$ is:
     \[
     g=1+\frac{C^{2}+C\cdot K_{\mathbb{F}_{n}}}{2}=1+ab-a-b+\frac{an(1-a)}{2},
     \]
     so the geometric genus of $\widetilde{C}$ is $\widetilde{g}=g-\delta$.\\
     
     The diagram (\ref{DiagramaD}) becomes:
    \begin{equation}\label{diagramaestandar}
        \begin{tikzcd}
         \bigwedge^2 H^{0}(X, \mathcal{O}_{X}(K_{X}+\widetilde{C})) \arrow{rrr}{\Phi_{X,\mathcal{O}_{X}(K_{X}+\widetilde{C})}} \arrow[swap]{dd}{\textnormal{Res}} & & & H^{0}(X,\Omega^{1}_{X}(2K_{X}+2\widetilde{C})) \arrow{d}{\alpha} \arrow[dd, bend left=85, "H^{0}(\rho)"]
         \\%
         & & & H^{0}(\widetilde{C}, \Omega^{1}_{X}(2K_{X}+2\widetilde{C})\mid_{\widetilde{C}})  \arrow{d}{\beta} \\%
        \bigwedge^{2} H^{0}(\widetilde{C}, \Omega^{1}_{\widetilde{C}}) \arrow[swap]{rrr}{\Phi_{\widetilde{C}}} && & H^{0}(\widetilde{C}, (\Omega^{1}_{\widetilde{C}})^{\otimes 3})
       \end{tikzcd}
    \end{equation}

    \begin{proposition} \label{propositionres}
        The restriction maps $\textnormal{res: } H^{0}(X, \mathcal{O}_{X}(K_{X}+\widetilde{C}))\rightarrow H^{0}(\widetilde{C},\Omega^{1}_{\widetilde{C}})$ and $\textnormal{Res }=\bigwedge^{2}$res are surjective maps.
    \end{proposition} 
    \begin{proof}
    Tensoring the sequence (\ref{se1}) with $\mathcal{O}_{X}(K_{X}+\widetilde{C})$, we obtain the short exact sequence
    \begin{equation*}
    0\rightarrow \mathcal{O}_{X}(K_{X}) \rightarrow \mathcal{O}_{X}(K_{X}+\widetilde{C}) \rightarrow \Omega^{1}_{\widetilde{C}} \rightarrow 0,
    \end{equation*}
    we obtain the long exact sequence:
    \begin{equation*}
       0 \rightarrow H^{0}(X, \mathcal{O}_{X}(K_{X})) \rightarrow H^{0}(X,\mathcal{O}_{X}(K_{X}+\widetilde{C}))\xrightarrow{\textnormal{res}} H^{0}(\widetilde{C}, \Omega^{1}_{\widetilde{C}}) \rightarrow H^{1}(X,\mathcal{O}_{X}(K_{X})).
    \end{equation*}
    Since the irregularity of $X$, defined as $q(X)=h^{1}(X,\mathcal{O}_{X})$, is a birational invariant, then we have that $q(X) = q(\mathbb{F}_{n})  = h^{1}(\mathbb{F}_{n},\mathcal{O}_{\mathbb{F}_{n}}) = 0$ and, by Serre duality, $ H^{1}(X,\mathcal{O}_{X}) \cong  H^{1}(X,\mathcal{O}_{X}(K_{X}))=0$. Hence, both $\mathrm{res}$ and $\mathrm{Res}$ are surjective maps.
   \end{proof}

\begin{remark} \label{remarkcorank2}
By Proposition \ref{propositionres} and Remark \ref{remarkcorank1}, we have:
   \[
\textnormal{corank}\left(\Phi_{\widetilde{C}}\right)=\textnormal{corank}\left(H^{0}(\rho) \circ \Phi_{X, \mathcal{O}_{X}(K_{X}+\widetilde{C})}\right).
   \]
In particular, if $\Phi_{X,\mathcal{O}_{X}(K_{X}+\widetilde{C})}$ is surjective, then $\textnormal{corank}(\Phi_{\widetilde{C}})=\textnormal{corank}(H^{0}(\rho))$.
\end{remark}
We study these two maps to compute the corank of $H^{0}(\rho)=\beta \circ \alpha$, and then proving the surjectivity of $\Phi_{X,\mathcal{O}_{X}(K_{X}+\widetilde{C})}$ under certain conditions on the coefficients $a$ and $b$ that define the curve $C$.

    \subsection{Cohomological Computations}
     \label{subsec:corankbetaalpha}
  By the exact sequences (\ref{sucesioneuleromega}), (\ref{sucesiondualnormal}), (\ref{sucesionlogarithmic2}) and (\ref{sucesionlogarithmic3}), we also have the following commutative diagram:
  \begin{equation}\label{diagramaconmuativo1}
 \begin{tikzcd}[column sep=1.5em,row sep=1.5em]
 &  0\arrow{d}  & 0\arrow{d} &  & 
 \\%
& \Omega^{1}(2K_{X}+\widetilde{C}) \arrow{d} \arrow[r, equal] & \Omega^{1}(2K_{X}+\widetilde{C})\arrow{d} &  & 
 \\
0\arrow{r} & \Omega^{1}_{X}(\log \widetilde{C})(2K_{X}+\widetilde{C})\arrow{d} \arrow{r}  & \Omega^{1}(2K_{X}+2\widetilde{C})\arrow{d} \arrow{r}&(\Omega^{1}_{\widetilde{C}})^{\otimes 3})\arrow{r} \arrow[d, equal] & 0 
 \\
0\arrow{r} &  \mathcal{O}_{\widetilde{C}}(2K_{X}+\widetilde{C})\arrow{d}\arrow{r}  &\Omega^{1}(2K_{X}+2\widetilde{C})|_{\widetilde{C}}\arrow{d}\arrow{r} & (\Omega^{1}_{\widetilde{C}})^{\otimes 3}\arrow{r} & 0 
 \\
& 0   & 0 &  & 
\end{tikzcd}
\end{equation}
this diagram is the dual to the diagram on (\cite{MR2247603}, 3.56) but here he use \textit{the sheaf of germs of tangent vectors to $X$ which are tangent to $\widetilde{C}$}.\\

First, we compute the cohomology groups appearing in the standard diagram (\ref{diagramaestandar}) and establish conditions for the vanishing of certain cohomology groups in diagram (\ref{diagramaconmuativo1}).
     Consider a blow-up diagram
    \[
      \begin{tikzcd}
       \widetilde{C} \arrow[r, hookrightarrow] \arrow{d}{\varphi=\sigma|_{\widetilde{C}}}  &  X\arrow{d}{\sigma}\\ C\arrow[r, hookrightarrow] & \mathbb{F}_{n}
      \end{tikzcd}
    \]
   Note that 
   \begin{equation*}
    \begin{split}
        d\sigma:T_{X}\rightarrow\sigma^{*}T_{\mathbb{F}_{n}}
    \end{split}
   \end{equation*}
   is an isomorphism off $E$. By (\cite{MR1644323}, Lemma 15.4), there are short exact sequences
   \begin{equation}\label{setang}
    \begin{split}
        0 \rightarrow T_{X}\rightarrow\sigma^{*}T_{\mathbb{F}_{n}}\rightarrow \mathcal{Q}\rightarrow 0,    \end{split}
   \end{equation}
   where the sheaf $\mathcal{Q}$ is supported on $E$, and fits into the following sequence:
   \begin{equation*}
        \begin{split}
        0 \rightarrow \mathcal{O}_{E}(-1)\rightarrow\mathcal{O}_{E}^{\oplus2}\rightarrow \mathcal{Q}\rightarrow 0,    
        \end{split}
   \end{equation*}
   with $E=E_{1}+\cdots +E_{\delta}$. Then, we have $\mathcal{Q}\cong \mathcal{O}_{E}(1)$. Dualizing the sequence $(\ref{setang})$, we get a short exact sequence of sheaves on $X$
   \begin{equation}\label{secotang}
    \begin{split}
        0 \rightarrow \sigma^{*}\Omega^{1}_{\mathbb{F}_{n}}\rightarrow\Omega^{1}_{X}\rightarrow \Omega^{1}_{X|\mathbb{F}_{n}}\rightarrow 0,    \end{split}
   \end{equation}
   where $\Omega^{1}_{X|\mathbb{F}_{n}}$ is the relative cotangent sheaf, and $\Omega^{1}_{X|\mathbb{F}_{n}}\cong \mathcal{E}xt^{1}_{X}(\mathcal{Q},\mathcal{O}_{X})$. By (\cite{MR463157}, III Proposition 6.5) we have $\mathcal{E}xt^{1}_{X}(\mathcal{Q},\mathcal{O}_{X})\cong\mathcal{O}_{E}(-2)$, that is,  $\Omega^{1}_{X|\mathbb{F}_{n}}\cong\bigoplus_{j=1}^{\delta}\mathcal{O}_{E_{j}}(2E_{j})$.  For the sequence (\ref{secotang}), there is a long exact sequence of sheaves on $\mathbb{F}_{n}$:
   \begin{align*}
      0 &\to \sigma_{*}(\sigma^{*}\Omega^{1}_{\mathbb{F}_{n}}) \rightarrow \sigma_{*}\Omega^{1}_{X} \rightarrow \sigma_{*}\Omega^{1}_{X|\mathbb{F}_{n}} \rightarrow  \\
     &\to R^{1}\sigma_{*}(\sigma^{*}\Omega^{1}_{\mathbb{F}_{n}}) \rightarrow R^{1}\sigma_{*}\Omega^{1}_{X} \rightarrow R^{1}\sigma_{*}\Omega^{1}_{X|\mathbb{F}_{n}} \rightarrow  \\
    &\to R^{2}\sigma_{*}(\sigma^{*}\Omega^{1}_{\mathbb{F}_{n}}) \rightarrow R^{2}\sigma_{*}\Omega^{1}_{X}  \rightarrow R^{2}\sigma_{*}\Omega^{1}_{X|\mathbb{F}_{n}} \rightarrow  0.
    \end{align*}
   By the projection formula, we have
    \begin{equation*}
     R^{i}\sigma_{*}(\sigma^{*}\Omega^{1}_{\mathbb{F}_{n}})\cong \Omega^{1}_{\mathbb{F}_{n}}\otimes R^{i}\sigma_{*}\mathcal{O}_{X}.
    \end{equation*}
   Since $R^{i}\sigma_{*}\mathcal{O}_{X}=0$, for all $i>0$, it follows that
    \begin{equation*}
    R^{i}\sigma_{*}(\sigma^{*}\Omega^{1}_{\mathbb{F}_{n}})=0,
    \end{equation*}
    for all $i>0$. Then $R^{1}\sigma_{*}\Omega^{1}_{X} \cong R^{1}\sigma_{*}\Omega^{1}_{X|\mathbb{F}_{n}}$ and $R^{2}\sigma_{*}\Omega^{1}_{X} \cong R^{2}\sigma_{*}\Omega^{1}_{X|\mathbb{F}_{n}}$. As $\Omega^{1}_{X|\mathbb{F}_{n}}\cong\bigoplus_{j=1}^{\delta}\mathcal{O}_{E_{j}}(2E_{j})$, it follows from the definition of the higher direct image that
  \begin{equation}\label{pushfoward0}
    \begin{split}
        \sigma_{*}\Omega^{1}_{X|\mathbb{F}_{n}}& \cong \bigoplus_{j=1}^{\delta}H^{0}(E_{j},\mathcal{O}_{E_{j}}(2E_{j}) )\otimes\mathcal{O}_{p_{j}} \cong \bigoplus_{j=1}^{\delta}H^{0}(\mathbb{P}^{1},\mathcal{O}_{\mathbb{P}^{1}}(-2) )\otimes\mathcal{O}_{p_{j}}\cong 0,
    \end{split}
    \end{equation}
    \begin{equation} \label{pushfoward1}
    \begin{split} 
         R^{1}\sigma_{*}\Omega^{1}_{X|\mathbb{F}_{n}}& \cong \bigoplus_{j=1}^{\delta}H^{1}(E_{j},\mathcal{O}_{E_{j}}(2E_{j}) )\otimes\mathcal{O}_{p_{j}} \cong \bigoplus_{j=1}^{\delta}H^{1}(\mathbb{P}^{1},\mathcal{O}_{\mathbb{P}^{1}}(-2) )\otimes\mathcal{O}_{p_{j}}\cong \bigoplus_{j=1}^{\delta}\mathcal{O}_{p_{j}},
    \end{split}
    \end{equation}
    \begin{equation}\label{pushfoward2}
    \begin{split} 
      R^{1}\sigma_{*}\Omega^{1}_{X|\mathbb{F}_{n}}& \cong \bigoplus_{j=1}^{\delta}H^{2}(E_{j},\mathcal{O}_{E_{j}}(2E_{j}) )\otimes\mathcal{O}_{p_{j}} \cong \bigoplus_{j=1}^{\delta}H^{2}(\mathbb{P}^{1},\mathcal{O}_{\mathbb{P}^{1}}(-2) )\otimes\mathcal{O}_{p_{j}}\cong 0,
    \end{split}
   \end{equation}
  this implies that $\sigma_{*}\Omega^{1}_{X}\cong \sigma_{*},(\sigma^{*}\Omega^{1}_{\mathbb{F}_{n}})\cong\Omega^{1}_{\mathbb{F}_{n}}$, $R^{1}\sigma_{*}\Omega^{1}_{X}=\mathcal{O}_{Z}$ and $R^{2}\sigma_{*}\Omega^{1}_{X}=0$.\\
  Therefore, from (\ref{secotang}) we have the exact sequence of sheaves on $X$
   \begin{equation}
    \begin{split} \label{secotang2}
        0 \rightarrow \sigma^{*}\Omega^{1}_{\mathbb{F}_{n}}\rightarrow\Omega^{1}_{X}\rightarrow \bigoplus_{j=1}^{\delta}\mathcal{O}_{E_{j}}(2E_{j})\rightarrow 0.
    \end{split}
   \end{equation}
   \subsubsection*{Cohomology of $\Omega^{1}_{X}(2K_{X}+\widetilde{C})$:}
   By tensoring the sequence $(\ref{secotang2})$ by $\mathcal{O}_{X}(2K_{X}+\widetilde{C})$, we obtain
  \begin{equation}
    \begin{split} \label{secotang3}
        0 \rightarrow \sigma^{*}\Omega^{1}_{\mathbb{F}_{n}}(2K_{\mathbb{F}_{n}}+C)\rightarrow\Omega^{1}_{X}(2K_{X}+\widetilde{C})\rightarrow \bigoplus_{j=1}^{\delta}\mathcal{O}_{E_{j}}(2E_{j}+2K_{X}+\widetilde{C})\rightarrow 0,  
    \end{split}
   \end{equation}
    since $2K_{X}+\widetilde{C}=2(\sigma^{*}K_{\mathbb{F}_{n}}+E)+ (\sigma^{*}C-2E)=\sigma^{*}(2K_{\mathbb{F}_{n}}+C)$.\\
    Note that 
    \begin{equation*}
    \deg(\mathcal{O}_{E_{j}}(2E_{j}+2K_{X}+\widetilde{C}))=2E^{2}_{j}+2K_{X}.E_{j}+\widetilde{C}.E_{j}=-2,
     \end{equation*}
     then $H^{0}(E_{j},\mathcal{O}_{E_{j}}(2E_{j}+2K_{X}+\widetilde{C}))=0$ and $H^{1}(E_{j},\mathcal{O}_{E_{j}}(2E_{j}+2K_{X}+\widetilde{C}))\cong \mathbb{C}$, and the exact sequence induced in cohomology of the sequence (\ref{secotang3}) is
     \begin{align*}
        0 &\to H^0(X,\sigma^{*}\Omega^{1}_{\mathbb{F}_{n}}(2K_{\mathbb{F}_{n}}+C)) \rightarrow H^{0}(X,\Omega^{1}_{X}(2K_{X}+\widetilde{C})) \rightarrow 0 \rightarrow  \\
        &\to H^1(X,\sigma^{*}\Omega^{1}_{\mathbb{F}_{n}}(2K_{\mathbb{F}_{n}}+C)) \rightarrow H^{1}(X,\Omega^{1}_{X}(2K_{X}+\widetilde{C})) \rightarrow \mathbb{C}^{\delta} \rightarrow  \\
        &\to H^2(X,\sigma^{*}\Omega^{1}_{\mathbb{F}_{n}}(2K_{\mathbb{F}_{n}}+C)) \rightarrow H^{2}(X,\Omega^{1}_{X}(2K_{X}+\widetilde{C})) \rightarrow 0. 
    \end{align*}
     By the projection formula, we have that $R^{i}\sigma_{*}(\sigma^{*}\Omega^{1}_{\mathbb{F}_{n}}(2K_{\mathbb{F}_{n}}+C))\cong\Omega^{1}_{\mathbb{F}_{n}}(2K_{\mathbb{F}_{n}}+C)\otimes R^{i}\sigma_{*}\mathcal{O}_{X}=0$, for all $i>0$.\\ Then, by (\cite{MR463157}, III 8.1), there are natural isomorphism, for each $i\geq 0$, 
   \begin{equation*}
      \begin{split}
      H^{i}(X,\sigma^{*}\Omega^{1}_{\mathbb{F}_{n}}(2K_{\mathbb{F}_{n}}+C))&\cong H^{i}(\mathbb{F}_{n},\sigma_{*}(\sigma^{*}\Omega^{1}_{\mathbb{F}_{n}}(2K_{\mathbb{F}_{n}}+C))\\
      &\cong H^{i}(\mathbb{F}_{n},\Omega^{1}_{\mathbb{F}_{n}}(2K_{\mathbb{F}_{n}}+C))
    \end{split}
   \end{equation*}
   since $\sigma_{*}(\sigma^{*}\Omega^{1}_{\mathbb{F}_{n}}(2K_{\mathbb{F}_{n}}+C))\cong \Omega^{1}_{\mathbb{F}_{n}}(2K_{\mathbb{F}_{n}}+C)\otimes\sigma_{*}\mathcal{O}_{X}$ by projection formula, and $\sigma_{*}\mathcal{O}_{X}=\mathcal{O}_{\mathbb{F}_{n}}$.\\
   
    \begin{lemma} \label{lemma1}
       If $a\geq5$ and $b\geq(a-2)n+6$, or if $n=0$, $a\geq 5$, $b\geq 5$, then $H^{i}(\mathbb{F}_{n},\Omega^{1}_{\mathbb{F}_{n}}(2K_{\mathbb{F}_{n}}+C))=0$ for $i=1,2$.
   \end{lemma}
   \begin{proof}
       See \cite{MR1061775}, Lemma 3.6.
   \end{proof}
   To relate cohomology on $X$ to cohomology on $\mathbb{F}_{n}$, we analyze the behavior under blow-up.
   \begin{lemma}\label{condiciones1}
     Suppose that $a\geq 5$ and $b\geq (a-2)n+6$, or $a\geq 5$ and $b\geq5$ for $n=0$. Then
      \begin{itemize}
        \item $H^{0}(X,\Omega^{1}_{X}(2K_{X}+\widetilde{C}))\cong H^{0}(\mathbb{F}_{n},\Omega^{1}_{\mathbb{F}_{n}}(2K_{\mathbb{F}_{n}}+C))$,
        \item $H^{1}(X,\Omega^{1}_{X}(2K_{X}+\widetilde{C}))\cong \mathbb{C}^{\delta}$,
        \item $H^{2}(X,\Omega^{1}_{X}(2K_{X}+\widetilde{C}))=0$.
     \end{itemize}
     \end{lemma}
     \begin{proof}
         This follows from the exact sequence induced in cohomology of the sequence (\ref{secotang3}) and by Lemma \ref{lemma1}.
     \end{proof}
     \subsubsection*{Cohomology of $\Omega^{1}_{X}(2K_{X}+2\widetilde{C})$:}
     Now, by tensoring the sequence (\ref{secotang2}) by $\mathcal{O}_{X}(2K_{X}+2\widetilde{C})$, we obtain 
     \begin{equation}
     \begin{split} \label{secotang4}
        0 \rightarrow \sigma^{*}\Omega^{1}_{\mathbb{F}_{n}}(2K_{\mathbb{F}_{n}}+2C)\otimes\mathcal{O}_{X}(-2E)\rightarrow\Omega^{1}_{X}(2K_{X}+2\widetilde{C})\rightarrow \bigoplus_{j=1}^{\delta}\mathcal{O}_{E_{j}}(2E_{j}+2K_{X}+2\widetilde{C})\rightarrow 0,  
    \end{split}
    \end{equation}
     since $2K_{X}+2\widetilde{C}=2(\sigma^{*}K_{\mathbb{F}_{n}}+E)+ 2(\sigma^{*}C-2E)=\sigma^{*}(2K_{\mathbb{F}_{n}}+2C)-2E$.\\
     Note that 
     \begin{equation*}
       \deg (\mathcal{O}_{E_{j}}(2E_{j}+2K_{X}+2\widetilde{C}))=2E_{j}^{2}+2K_{X}.E_{j}+2\widetilde{C}.E_{j}=0,
     \end{equation*}
       thus $H^{0}(E_{j},\mathcal{O}_{E_{j}}(2E_{j}+2K_{X}+2\widetilde{C}))\cong \mathbb{C}$ and $H^{1}(E_{j},\mathcal{O}_{E_{j}}(2E_{j}+2K_{X}+2\widetilde{C}))=0$, and the sequence induced in cohomology of (\ref{secotang4}) is
      \begin{align*}
          0 &\to H^0(X,\sigma^{*}\Omega^{1}_{\mathbb{F}_{n}}(2K_{\mathbb{F}_{n}}+2C)\otimes \mathcal{O}_{X}(-2E)) \rightarrow H^{0}(X,\Omega^{1}_{X}(2K_{X}+2\widetilde{C})) \rightarrow \mathbb{C}^{\delta} \rightarrow  \\
         &\to H^1(X,\sigma^{*}\Omega^{1}_{\mathbb{F}_{n}}(2K_{\mathbb{F}_{n}}+2C)\otimes \mathcal{O}_{X}(-2E)) \rightarrow H^{1}(X,\Omega^{1}_{X}(2K_{X}+2\widetilde{C})) \rightarrow 0 \rightarrow  \\
        &\to H^2(X,\sigma^{*}\Omega^{1}_{\mathbb{F}_{n}}(2K_{\mathbb{F}_{n}}+2C)\otimes \mathcal{O}_{X}(-2E)) \rightarrow H^{2}(X,\Omega^{1}_{X}(2K_{X}+2\widetilde{C})) \rightarrow 0. 
     \end{align*}
     Again, by the projection formula 
     \begin{equation*}
         R^{i}\sigma_{*}(\sigma^{*}\Omega^{1}_{\mathbb{F}_{n}}(2K_{\mathbb{F}_{n}}+2C)\otimes\mathcal{O}_{X}(-2E))\cong\Omega^{1}_{\mathbb{F}_{n}}(2K_{\mathbb{F}_{n}}+C)\otimes R^{i}\sigma_{*}\mathcal{O}_{X}(-2E)=0
     \end{equation*}
     since $R^{i}\sigma_{*}\mathcal{O}_{X}(-2E)=0$, for all $i>0$. Therefore, there are natural isomorphism, for each $i\geq0$,
    \begin{equation*}
       \begin{split}
           H^{i}(X,\sigma^{*}\Omega^{1}_{\mathbb{F}_{n}}(2K_{\mathbb{F}_{n}}+2C)\otimes\mathcal{O}_{X}(-2E))&\cong H^{i}(\mathbb{F}_{n},\sigma_{*}(\sigma^{*}\Omega^{1}_{\mathbb{F}_{n}}(2K_{\mathbb{F}_{n}}+2C)\otimes\mathcal{O}_{X}(-2E)))\\
          &\cong H^{i}(\mathbb{F}_{n},\Omega^{1}_{\mathbb{F}_{n}}(2K_{\mathbb{F}_{n}}+2C)\otimes \mathcal{I}^{2}_{Z})
      \end{split}
     \end{equation*}
     since $\sigma_{*}(\sigma^{*}\Omega^{1}_{\mathbb{F}_{n}}(2K_{\mathbb{F}_{n}}+2C)\otimes\mathcal{O}_{X}(-2E))\cong\Omega^{1}_{\mathbb{F}_{n}}(2K_{\mathbb{F}_{n}}+2C)\otimes\sigma_{*}\mathcal{O}_{X}(-2E)$ and $\sigma_{*}\mathcal{O}_{X}(-2E)=\mathcal{I}^{2}_{Z}$, where $\mathcal{I}_{Z}$ is the ideal sheaf of $Z$ on $\mathbb{F}_{n}$ and $\mathcal{I}^{2}_{Z}$ is the product $\mathcal{I}_{Z}\cdot \mathcal{I}_{Z}$.\\
     Let $\mathcal{F}=\Omega^{1}_{\mathbb{F}_{n}}(2K_{\mathbb{F}_{n}}+2C)$, we use the exact sequence
        \begin{equation}
           0\rightarrow \mathcal{I}^{2}_{Z}\rightarrow \mathcal{O}_{\mathbb{F}_{n}}\rightarrow \mathcal{O}_{\mathbb{F}_{n}}/\mathcal{I}^{2}_{Z}\rightarrow 0.
        \end{equation}
      By tensoring this sequence with $\mathcal{F}$, we have the long exact sequence in cohomology
       \begin{align*}
          0 &\to H^{0}(\mathbb{F}_{n},\mathcal{F}\otimes \mathcal{I}^{2}_{Z}) \rightarrow H^{0}(\mathbb{F}_{n},\mathcal{F}) \xrightarrow{h} H^0(\mathbb{F}_{n},\mathcal{F}\otimes \left(\mathcal{O}_{\mathbb{F}_{n}}/\mathcal{I}^{2}_{Z}\right)) \rightarrow  \\
          &\to H^{1}(\mathbb{F}_{n},\mathcal{F}\otimes \mathcal{I}^{2}_{Z}) \rightarrow H^{1}(\mathbb{F}_{n},\mathcal{F}) \rightarrow 0 \rightarrow  \\
          &\to H^{2}(\mathbb{F}_{n},\mathcal{F}\otimes \mathcal{I}^{2}_{Z}) \rightarrow H^{2}(\mathbb{F}_{n},\mathcal{F}) \rightarrow 0. 
      \end{align*}
     \begin{lemma}\label{lemma3}
         Suppose that $a\geq 3$ and $b\geq (a-1)n+3$, then $H^{i}(\mathbb{F}_{n},\Omega^{1}_{\mathbb{F}_{n}}(2K_{\mathbb{F}_{n}}+2C))=0$ for $i=1,2$.
      \end{lemma}
      \begin{proof}
           See \cite{MR1061775}, Lemma 3.6.
      \end{proof}
    \begin{remark} \label{remarkvanishing}
      If the conditions of Lemma \ref{lemma3} are satisfied and $h$ is surjective, then
        \begin{equation*}
        H^{1}(\mathbb{F}_{n},\mathcal{F}\otimes \mathcal{I}^{2}_{Z})\cong H^{1}(X,\sigma^{*}\Omega^{1}_{\mathbb{F}_{n}}(2K_{\mathbb{F}_{n}}+2C)\otimes \mathcal{O}_{X}(-2E))=0.
      \end{equation*}
   \end{remark}
   To analyze the surjectivity of the morphism $h$, we first establish the relationship between the ideal sheaf of $Z$ and a tensor product of maximal ideal sheaves. Then we introduce the definitions of \textit{k-jet spanned} at a point $p$ and \textit{k-jet ample} at $Z$.
     \begin{lemma}\label{lemma2}
    Suppose that $Z=\{p_{1}\cdots, p_{\delta}\}$ is a zero-dimensional subscheme of $\mathbb{F}_{n}$ formed by $\delta$ distinct reduced closed points. For each $p_{j}$, let $\mathfrak{m}_{p_{j}}$ be the maximal ideal sheaf of $p_{j}$ in $\mathbb{F}_{n}$, i.e., the stalks of $\mathfrak{m}_{p_{j}}$ at a point $y\neq p_{j}$ is $\mathcal{O}_{\mathbb{F}_{n},y}$ and at $p_{j}$ is the maximal ideal  $\mathfrak{m}_{p_{j}}\mathcal{O}_{\mathbb{F}_{n},p_{j}}\subset\mathcal{O}_{\mathbb{F}_{n},p_{j}}$. Then 
    \begin{equation*}
        \mathcal{I}_{Z}\cong \bigotimes_{j=1}^{\delta} \mathfrak{m}_{p_{j}}.
    \end{equation*}
    Similarly, $\mathcal{I}^{2}_{Z}\cong \bigotimes_{j=1}^{\delta} \mathfrak{m}^{2}_{p_{j}}$.
      \end{lemma}
      \begin{proof}
    This demonstration is more general. Let $Y$ a smooth algebraic variety of dimension two. For each point $x$ on $Y$ let $\mathfrak{m}_{x}$ be the maximal ideal sheaf of $x$ in $Y$, i.e., the stalk of $\mathfrak{m}_{x}$ at a point $y\neq x$ is $\mathcal{O}_{Y,y}$ and at $x$ is the maximal ideal $\mathfrak{m}_{x}\mathcal{O}_{Y,x}\subset \mathcal{O}_{Y,x}$.\\
    Let $Z:=\{p_{1},\cdots,p_{\delta}\}$ be $\delta$ distinct reduced closed points, and consider the inclusion $i:= Z\hookrightarrow Y$ and the morphism $i^{\#}:\mathcal{O}_{Y}\rightarrow i_{*}\mathcal{O}_{Z}$. Note that $\mathcal{I}_{Z}=\ker i^{\#}$. \\
    Observe that, for every $x\in Y$, we have
    \begin{equation*}
        i^{\#}_{x}=\left\{ \begin{array}{lcc} i^{\#}_{x}:\mathcal{O}_{Y,x}\rightarrow0 & if & x\notin Z, \\  \\ i^{\#}_{x}:\mathcal{O}_{Y,x}\rightarrow (i_{*}\mathcal{O}_{Z})_{x} & if & x\in Z. \end{array} \right.
    \end{equation*}
    For the case where $x\in Z$, we have that $\mathcal{O}_{Y,x}$ is a local ring of Krull dimension two, while $\mathcal{O}_{Z,x}$ is a local ring of Krull dimension zero. Since $Z$ is reduced subscheme, $\mathcal{O}_{Z,x}$ has no nilpotent elements. Furthermore, every element on the maximal ideal of $\mathcal{O}_{Z,x}$ is nilpotent which implies that the maximal ideal is zero. Therefore, as $\ker i^{\#}_{x}$
    is precisely the maximal ideal of $\mathcal{O}_{Y,x}$.\\
    Therefore the stalks of $\mathcal{I}_{Z}$ coincides with the stalks of $\mathfrak{m}_{p_{1}}\otimes\mathfrak{m}_{p_{2}}\otimes \cdots\otimes \mathfrak{m}_{p_{\delta}}$, that is, at a point $y\neq p_{j}$ is $\mathcal{O}_{Y,y}$ and at $p_{j}$ is the maximal ideal $\mathfrak{m}_{p_{j}}\mathcal{O}_{Y,p_{j}}\subset \mathcal{O}_{Y,p_{j}}$. \\
      \end{proof}
   \begin{definition}
   A vector bundle $\mathcal{E}$ on $\mathbb{F}_{n}$ is \textit{k- jet spanned} at $p$ if the evaluation map
   \begin{equation*}
          X\times H^{0}(\mathbb{F}_{n},\mathcal{E})\rightarrow H^{0}(\mathbb{F}_{n},\mathcal{E}\otimes(\mathcal{O}_{\mathbb{F}_{n}}/\mathfrak{m}_{p_{j}}^{k+1}))
   \end{equation*}
   is surjective. We say that a vector bundle $\mathcal{E}$ on $\mathbb{F}_{n}$ is a \textit{k-jet ample} at $Z$ if for every $\delta-$tuple $(k_{1},\cdots,k_{\delta})$ of positive integers such that $\sum_{i=1}^{\delta}k_{i}=k+1$, the evaluation map
     \begin{equation*}
        X\times H^{0}(\mathbb{F}_{n},\mathcal{E})\rightarrow H^{0}(\mathbb{F}_{n},\mathcal{E}\otimes(\mathcal{O}_{\mathbb{F}_{n}}/\otimes_{i=1}^{\delta}\mathfrak{m}_{p_{j}}^{k_{i}}))
      \end{equation*}
    is surjective.
    \end{definition}
    These definitions are given in a more general context, namely for vector bundles on a smooth variety; for further details, see \cite{MR1698897}. Hence in particular $k-$ jet ample implies $k-$ jet spanned. Note that $\mathcal{E}$ is 0-jet ample if only if $\mathcal{E}$ is 0-jet spanned, if only if $\mathcal{E}$ is spanned by its global sections. Moreover, for a line bundle $L$,  $L$ is 1-jet spanned if only if $L$ is base-point-free; and $L$ is 1-jet ample if only if $L$ is very ample. \\

    For the case of $\delta=1$, the morphism $h$ is $H^{0}(\mathbb{F}_{n},\mathcal{F})\rightarrow H^{0}(\mathbb{F}_{n},\mathcal{F}\otimes \left(\mathcal{O}_{\mathbb{F}_{n}}/\mathfrak{m}_{p}^{2}\right))$ according to the Lemma \ref{lemma2}. The next lemma guarantees that $h$ is surjective in the case where $\delta=1$.
    \begin{lemma}\label{lemma1nodo}
        Let $p\in \mathbb{F}_{n}$. Suppose that $a\geq 3$ and $b\geq (a-1)n+2$ then $\mathcal{F}=\Omega^{1}_{\mathbb{F}_{n}}(2K_{\mathbb{F}_{n}}+2C)$ is a $1-$jet spanned at $p$.
    \end{lemma}
    \begin{proof}
        Consider $\mathcal{L}=\mathcal{O}_{\mathbb{F}_{n}}(C_{0}+ nF)$ a line bundle on $\mathbb{F}_{n}$. Observe that $\mathcal{L}$ is 1-jet spanned since $\mathcal{L}$ is base-point-free (recall that a line bundle $\mathcal{O}_{\mathbb{F}_{n}}(r_{1}C_{0}+ r_{2}F)$ on $\mathbb{F}_{n}$ is base-point-free if only if $r_{2}\geq r_{1}\,\cdot n$). Note that $\mathcal{F}\otimes \mathcal{L}^{-1}=\Omega^{1}_{\mathbb{F}_{n}}((2a-5)C_{0}+(2b-4-3n)F)$ is generated globally if $2a-5\geq 0$ and $2b-4-3n\geq (2a-5)n$. Thus, with the conditions, we see that $\mathcal{F}\otimes\mathcal{L}^{-1}$ is 0-jet spanned and $\mathcal{L}$ is $1-$jet spanned. Therefore, by (\cite{MR1698897}, 2.3), $\mathcal{F}=\mathcal{F}\otimes\mathcal{L}^{-1}\otimes\mathcal{L}$ is $1-$jet spanned.
    \end{proof}
    In general, we have 
    \begin{lemma}\label{lemmadeltanodos}
         Suppose that $a\geq \delta +2$ and $b\geq(a-1)n+\delta+2$. Then $\mathcal{F}$ is a ($2\delta-1$)-jet ample at $Z$. In particular, $h$ is surjective.
    \end{lemma}
    \begin{proof}
        Let $\mathcal{L}=\mathcal{O}_{\mathbb{F}_{n}}((2\delta-1)C_{0}+(2\delta n+2\delta-n-1)F)$. As $\mathcal{O}_{\mathbb{F}_{n}}(C_{0}+(n+1)F)$ is $1-$jet ample (recall that a line bundle $\mathcal{O}_{\mathbb{F}_{n}}(r_{1}C_{0}+ r_{2}F)$ on $\mathbb{F}_{n}$ is very ample if only if $r_{2}> r_{1}\,\cdot n$) then $\mathcal{L}$ is $(2\delta-1)-$jet ample. Note that $\mathcal{F}\otimes \mathcal{L}^{-1}=\Omega^{1}_{\mathbb{F}_{n}}((2a-3-2\delta)C_{0}+(2b-3-n-2\delta n-2\delta)F)$ is generated globally if $2a-3-2\delta\geq 0$ and $2b-3-n-2\delta n-2\delta\geq (2a-3-2\delta)n$. Thus, with the conditions, we see that $\mathcal{F}\otimes \mathcal{L}^{-1}$ is 0-jet ample and $ \mathcal{L}$ is $(2\delta-1)$-jet ample. Therefore $\mathcal{F}=\mathcal{F}\otimes \mathcal{L}^{-1}\otimes\mathcal{L}$ is $(2\delta-1)$-jet ample.\\
    In particular, for the $\delta-$tuple $(2,2,\cdots, 2)$, the evaluation map 
    \begin{equation*}
        H^{0}(\mathbb{F}_{n},\mathcal{F})\rightarrow H^{0}(\mathbb{F}_{n},\mathcal{F}\otimes(\mathcal{O}_{\mathbb{F}_{n}}/\otimes_{i=1}^{\delta}\mathfrak{m}_{p_{i}}^{2}))
    \end{equation*}
   is surjective, and so is the morphism $h$.
   \end{proof}
   \begin{corollary}\label{corolariovanishing2k+2c}
    If $a\geq \delta +2$ and $b\geq (a-1)n+\delta+2$, then $H^{1}(X,\Omega^{1}_{X}(2K_{X}+2\widetilde{C}))=0.$
\end{corollary}
   \begin{proof}
        By Lemma \ref{lemma3} and Lemma \ref{lemmadeltanodos}.
   \end{proof}
   
 In summary, putting these all together, we obtain the following:
\begin{proposition}
    Suppose that $a\geq \max\{5,\delta+2\}$ and $b\geq\max\{(a-2)n+6,(a-1)n+\delta+2\}$, or that $n=0$, $a\geq \max\{5,\delta+2\}$ and $b\geq \max\{5,\delta+2\}$. Then
      \begin{itemize}
        \item $H^{0}(X,\Omega^{1}_{X}(2K_{X}+\widetilde{C}))\cong \mathbb{C}^{8g+22-3C^{2}}$,
        \item $H^{1}(X,\Omega^{1}_{X}(2K_{X}+\widetilde{C}))\cong \mathbb{C}^{\delta}$,
        \item $H^{2}(X,\Omega^{1}_{X}(2K_{X}+\widetilde{C}))=0$,
        \item $H^{0}(X,\Omega^{1}_{X}(2K_{X}+2\widetilde{C}))\cong\mathbb{C}^{16g-4C^{2}-5\delta+14}$,
        \item $H^{1}(X,\Omega^{1}_{X}(2K_{X}+2\widetilde{C}))\cong H^{2}(X,\Omega^{1}_{X}(2K_{X}+2\widetilde{C}))=0,$
        \item $H^{0}(\widetilde{C},\Omega^{1}_{X}(2K_{X}+2\widetilde{C})|_{\widetilde{C}})\cong\mathbb{C}^{8g-C^{2}-4\delta-8}$,
        \item $H^{1}(\widetilde{C},\Omega^{1}_{X}(2K_{X}+2\widetilde{C})|_{\widetilde{C}})=0$.
    \end{itemize}
\end{proposition}
\begin{proof}
We will show that, for every divisor $\widetilde{D}$ in $X$, the holomorphic Euler characteristic $\chi(\Omega^{1}_{X}(\widetilde{D}))$ for the vector bundle $\Omega^{1}_{X}(\widetilde{D})$ is equal to $\widetilde{D}^{2}-(2+\delta)$; and by the hypotheses we can apply the lemma \ref{condiciones1} and corollary \ref{corolariovanishing2k+2c}.\\
Recall that $\sigma:X=\textnormal{Bl}_{Z}\mathbb{F}_{n}\rightarrow\mathbb{F}_{n}$ is the blow up of $\mathbb{F}_{n}$ at $Z=\{p_{1},\cdots,p_{\delta}\}$, and the Noether formula for surfaces is
\begin{equation*}
    \chi(\mathcal{O}_{\mathbb{F}_{n}})=\frac{1}{12}(K_{\mathbb{F}_{n}}^{2}+\chi_{top}(\mathbb{F}_{n}))=\frac{1}{12}(c_{1}^{2}+c_{2})
\end{equation*}
with the Chern class $c_{1}=c_{1}(T_{\mathbb{F}_{n}})=-c_{1}(\Omega^{1}_{\mathbb{F}_{n}})$ and $c_{2}=c_{2}(T_{\mathbb{F}_{n}})=c_{2}(\Omega^{1}_{\mathbb{F}_{n}})$. Given that $\chi(\mathcal{O}_{\mathbb{F}_{n}})=1$, then $\chi_{top}(\mathbb{F}_{n})=4$, and by the birational invariant $\chi$, we have $\chi(\mathcal{O}_{X})=\chi(\mathcal{O}_{\mathbb{F}_{n}})$ and $\chi(\mathcal{O}_{X})=\frac{1}{12}(K_{X}^{2}+\chi_{top}(X))=\frac{1}{12}(K_{\mathbb{F}_{n}}^{2}-\delta+\chi_{top}(X))$, this tells us that $\chi_{top}(X)=\chi_{top}(\mathbb{F}_{n})+\delta= 4+\delta$. The Riemann-Roch theorem for a rank 2 bundle $\mathcal{E}$ on the surface $X$ tell us that
    \begin{equation*}
        \chi(\mathcal{E})=2\chi(\mathcal{O}_{X})-\frac{1}{2}c_{1}(E)K_{X}+\frac{(c_{1}(\mathcal{E})^{2}-2c_{2}(\mathcal{E}))}{2}
    \end{equation*}
    where $c_{i}(\mathcal{E})$ are the Chern class of $\mathcal{E}$. If $\mathcal{E}=\Omega^{1}_{X}\otimes\mathcal{O}_{X}(\widetilde{D})$ for a divisor $\widetilde{D}$ on $X$, then $c_{1}(\Omega^{1}_{X}(\widetilde{D}))=c_{1}(\Omega^{1}_{X})+2c_{1}(\mathcal{O}_{X}(\widetilde{D}))=K_{X}+2\widetilde{D}$ and $c_{2}(\Omega^{1}_{X}(\widetilde{D}))=c_{2}(\Omega^{1}_{X})+c_{1}(\Omega^{1}_{X})c_{1}(\mathcal{O}_{X}(\widetilde{D}))+ (c_{1}(\mathcal{O}_{X}(\widetilde{D})))^{2}=4+\delta+K_{X}.\widetilde{D}+\widetilde{D}^{2}$.
    We then compute
    \begin{equation*}
        \begin{split}
            \chi(\Omega^{1}_{X}(\widetilde{D}))&=2-\frac{1}{2}(K_{X}+2\widetilde{D})K_{X}+\frac{(K_{X}+2\widetilde{D})^{2}-2(4+\delta+K_{X}.\widetilde{D}+\widetilde{D}^{2})}{2}\\
            &=\widetilde{D}^{2}-(2+\delta).
        \end{split}
    \end{equation*}
\end{proof}
Therefore, the diagram in (\ref{diagramaconmuativo1}) is as follows:
\[ 
\begin{tikzcd}[column sep=1em,row sep=1em]
   0\arrow{d}  & 0\arrow{d} &  & &
 \\%
 H^{0}(\Omega^{1}(2K_{X}+\widetilde{C}))\arrow{d} \arrow[r, equal] & H^{0}(\Omega^{1}(2K_{X}+\widetilde{C}))\arrow{d} &  &  &
 \\
 H^{0}(\Omega^{1}_{X}(\log \widetilde{C})(2K_{X}+\widetilde{C}))\arrow{d} \arrow[r, hook]  & H^{0}(\Omega^{1}(2K_{X}+2\widetilde{C}))\arrow{d}{\alpha} \arrow{r}{H^{0}(\rho)} & H^{0}((\Omega^{1}_{\widetilde{C}})^{\otimes 3})\arrow{r} \arrow[d, equal] & H^{1}(\Omega^{1}_{X}(\log \widetilde{C})(2K_{X}+\widetilde{C}))\arrow{r}\arrow{d}& 0
 \\
  H^{0}(\mathcal{O}_{\widetilde{C}}(2K_{X}+\widetilde{C}))\arrow{d}\arrow[r, hook]  & H^{0}(\Omega^{1}(2K_{X}+\widetilde{C})|_{\widetilde{C}})\arrow{d}\arrow{r}{\beta} & H^{0}((\Omega^{1}_{\widetilde{C}})^{\otimes 3})\arrow{r} & H^{1}(\mathcal{O}_{\widetilde{C}}(2K_{X}+\widetilde{C})) \arrow{r}\arrow{d} & 0
 \\
 H^{1}(\Omega^{1}_{X}(2K_{X}+\widetilde{C}))\arrow{d}\arrow[r, equal]   & H^{1}(\Omega^{1}_{X}(2K_{X}+\widetilde{C}))\arrow{d} &  & 0 &
 \\
 H^{1}(\Omega^{1}_{X}(\log \widetilde{C})(2K_{X}+\widetilde{C}))\arrow{d}  & 0 &  & &
 \\
   H^{1}(\mathcal{O}_{\widetilde{C}}(2K_{X}+\widetilde{C}))\arrow{d} &  &  & &
 \\
   0 &  &  & &
\end{tikzcd}
\]
\begin{remark}\label{remarkh0rho1}
This diagram tells us that 
\[
\textnormal{corank}(H^{0}(\rho))= h^{1}(X,\Omega^{1}_{X}(\log\widetilde{C})(2K_{X}+\widetilde{C})).
\]
\end{remark}
 The next proposition gives us conditions to compute its dimension.
\begin{proposition}\label{propisomlogandh1}
    Suppose that $a\geq 5$ and $b\geq (a-2)n+6$, or that $a\geq 5$ and $b\geq 5$ if $n=0$. Then 
    \begin{equation*}
        H^{1}(X,\Omega^{1}_{X}(\log \widetilde{C})(2K_{X}+\widetilde{C}))\cong H^{1}(C,\mathcal{O}_{\mathbb{F}_{n}}(2K_{\mathbb{F}_{n}}+C)|_{C}).
    \end{equation*}
\end{proposition}
\begin{proof}
    The exact sequence of (\cite{MR1193913},2.3a) in our case is as follows:
    \begin{equation}
        0\rightarrow \Omega^{1}_{X}\rightarrow\Omega^{1}_{X}(\log \widetilde{C})\rightarrow\mathcal{O}_{\widetilde{C}}\rightarrow0,
    \end{equation}
    there is a long exact sequence of sheaves on $\mathbb{F}_{n}$
    \begin{equation}\label{pushfowardexactsequence}
        \begin{split}
            0 & \rightarrow \sigma_{*}\Omega^{1}_{X}\rightarrow \sigma_{*}\Omega^{1}_{X}(\log \widetilde{C})\rightarrow \sigma_{*}\mathcal{O}_{\widetilde{C}}\rightarrow\\
             &\rightarrow R^{1}\sigma_{*}\Omega^{1}_{X}\rightarrow R^{1}\sigma_{*}\Omega^{1}_{X}(\log \widetilde{C})\rightarrow R^{1}\sigma_{*}\mathcal{O}_{\widetilde{C}}\rightarrow\\
             &\rightarrow R^{2}\sigma_{*}\Omega^{1}_{X}\rightarrow R^{2}\sigma_{*}\Omega^{1}_{X}(\log \widetilde{C})\rightarrow 0.\\
        \end{split}
    \end{equation}
    From (\ref{pushfoward0}) we deduce $\sigma_{*}\Omega^{1}_{X}\cong\Omega^{1}_{\mathbb{F}_{n}}$, and from (\ref{pushfoward1}) and (\ref{pushfoward2}) we have $R^{1}\sigma_{*}\Omega^{1}_{X}\cong\mathcal{O}_{Z}$ and $R^{2}\sigma_{*}\Omega^{1}_{X}\cong 0$.\\ Also, we obtain that $R^{1}\sigma_{*}\Omega^{1}_{X}(\log\widetilde{C})=0$, thus $R^{1}\sigma_{*}\mathcal{O}_{\widetilde{C}}=0$.\\
    Therefore, the sequence in (\ref{pushfowardexactsequence}) is
    \begin{equation}\label{pushforwardexactsequence2}
        0\rightarrow \Omega^{1}_{\mathbb{F}_{n}} \rightarrow \sigma_{*}\Omega^{1}_{X}(\log\widetilde{C}) \rightarrow \sigma_{*}\mathcal{O}_{\widetilde{C}}\rightarrow\mathcal{O}_{Z}\rightarrow 0.
    \end{equation}
    Let $\mathcal{G}:=\textnormal{Im}(\sigma_{*}\Omega^{1}_{X}(\log\widetilde{C}) \rightarrow \sigma_{*}\mathcal{O}_{\widetilde{C}})$ and we obtain two short exact sequences 
    \begin{equation}\label{sucesioncortapushfoward1}
        0\rightarrow \Omega^{1}_{\mathbb{F}_{n}} \rightarrow \sigma_{*}\Omega^{1}_{X}(\log\widetilde{C}) \rightarrow \mathcal{G}\rightarrow 0,
    \end{equation}
    \begin{equation}\label{sucesioncortapushfoward2}
      0 \rightarrow \mathcal{G} \rightarrow \sigma_{*}\mathcal{O}_{\widetilde{C}} \rightarrow \mathcal{O}_{Z} \rightarrow 0.
    \end{equation}
    Now, from the normalization sequence on $C$
    \begin{equation}\label{exactsequencenormalization1}
        0 \rightarrow \mathcal{O}_{C}\rightarrow \varphi_{*}\mathcal{O}_{\widetilde{C}}\rightarrow\mathcal{O}_{Z}\rightarrow 0,
    \end{equation}
    where $\varphi: \widetilde{C}\rightarrow C$ is the normalization map of $C$. Note that, if $j:C\hookrightarrow \mathbb{F}_{n}$ is the inclusion then $\sigma|_{\widetilde{C}}=j\,\circ\, \varphi$, and we have $\sigma_{*}\mathcal{O}_{\widetilde{C}}=(j\,\circ\,\varphi)_{*}\mathcal{O}_{\widetilde{C}}=j_{*}(\varphi_{*}\mathcal{O}_{\widetilde{C}})$. Applying $j_{*}$ to (\ref{exactsequencenormalization1}), since $j_{*}$ is exact, we obtain an exact sequence on $\mathbb{F}_{n}$:
    \begin{equation}\label{exactsequencenormalization2}
         0 \rightarrow j_{*}\mathcal{O}_{C}\rightarrow j_{*}\varphi_{*}\mathcal{O}_{\widetilde{C}}\rightarrow\mathcal{O}_{Z}\rightarrow 0.
    \end{equation}
    Now, use the identity $j_{*}\varphi_{*}\mathcal{O}_{\widetilde{C}}=\sigma_{*}\mathcal{O}_{\widetilde{C}}$
    Thus, (\ref{exactsequencenormalization2}) becomes
    \begin{equation}\label{exactsequencenormalization3}
        0 \rightarrow \mathcal{O}_{C}\rightarrow\sigma_{*}\mathcal{O}_{\widetilde{C}}\rightarrow \mathcal{O}_{Z}\rightarrow 0.
    \end{equation}
    Comparing exact sequences (\ref{sucesioncortapushfoward2}) and (\ref{exactsequencenormalization3}), we obtain an isomorphism of sheaves on $\mathbb{F}_{n}$
    \begin{equation*}
        \mathcal{G}\cong \mathcal{O}_{C}.
    \end{equation*}
    Therefore, tensorizing the exact sequence in (\ref{sucesioncortapushfoward1}) with $\mathcal{O}_{\mathbb{F}_{n}}(2K_{\mathbb{F}_{n}}+C)$ we get
    \begin{equation*}
        0\rightarrow\Omega^{1}_{\mathbb{F}_{n}}(2K_{\mathbb{F}_{n}}+C)\rightarrow\sigma_{*}\Omega^{1}_{X}(\log\widetilde{C})(2K_{X}+\widetilde{C})\rightarrow\mathcal{O}_{C}\otimes\mathcal{O}_{\mathbb{F}_{n}}(2K_{\mathbb{F}_{n}}+C)\rightarrow 0
    \end{equation*}
    since $2K_{X}+\widetilde{C}=\sigma^{*}(2K_{\mathbb{F}_{n}}+C)$ and $\sigma_{*}\mathcal{O}_{X}(2K_{X}+\widetilde{C})=\mathcal{O}_{\mathbb{F}_{n}}(2K_{\mathbb{F}_{n}}+C)$. Thus, this sequence induces an exact sequence in cohomology
    \begin{equation*}
         \begin{split}
        0&\rightarrow H^{0}(\mathbb{F}_{n},\Omega^{1}_{\mathbb{F}_{n}}(2K_{\mathbb{F}_{n}}+C))\rightarrow H^{0}(\mathbb{F}_{n},\sigma_{*}\Omega^{1}_{X}(\log\widetilde{C})(2K_{X}+\widetilde{C}))\rightarrow H^{0}(C,\mathcal{O}_{\mathbb{F}_{n}}(2K_{\mathbb{F}_{n}}+C)|_{C})\rightarrow \\
        &\rightarrow H^{1}(\mathbb{F}_{n},\Omega^{1}_{\mathbb{F}_{n}}(2K_{\mathbb{F}_{n}}+C))\rightarrow H^{1}(\mathbb{F}_{n},\sigma_{*}\Omega^{1}_{X}(\log\widetilde{C})(2K_{X}+\widetilde{C}))\rightarrow H^{1}(C,\mathcal{O}_{\mathbb{F}_{n}}(2K_{\mathbb{F}_{n}}+C)|_{C})\rightarrow \\
         &\rightarrow H^{2}(\mathbb{F}_{n},\Omega^{1}_{\mathbb{F}_{n}}(2K_{\mathbb{F}_{n}}+C))\rightarrow H^{2}(\mathbb{F}_{n},\sigma_{*}\Omega^{1}_{X}(\log\widetilde{C})(2K_{X}+\widetilde{C}))\rightarrow 0. \\
         \end{split}
    \end{equation*}
    With the conditions of Lemma \ref{lemma1}, we obtain the vanishing of $H^{1}(\Omega^{1}_{\mathbb{F}_{n}}(2K_{\mathbb{F}_{n}}+C))$ and $H^{2}(\Omega^{1}_{\mathbb{F}_{n}}(2K_{\mathbb{F}_{n}}+C))$, and therefore we have the isomorphism
    \begin{equation*}
        H^{1}(\mathbb{F}_{n},\sigma_{*}\Omega^{1}_{X}(\log\widetilde{C})(2K_{X}+\widetilde{C}))\cong H^{1}(C,\mathcal{O}_{\mathbb{F}_{n}}(2K_{\mathbb{F}_{n}}+C)|_{C}).
    \end{equation*}
    Now, by the projection formula $R^{1}\sigma_{*}(\Omega^{1}_{X}(\log\widetilde{C})(2K_{X}+\widetilde{C}))\cong R^{1}\sigma_{*}\Omega^{1}(\log\widetilde{C})\otimes\mathcal{O}_{\mathbb{F}_{n}}(2K_{\mathbb{F}_{n}}+C)$, then $R^{1}\sigma_{*}(\Omega^{1}_{X}(\log\widetilde{C})(2K_{X}+\widetilde{C}))=0$ and therefore there is an isomorphism
    \begin{equation*}
        H^{1}(X,\Omega^{1}_{X}(\log\widetilde{C})(2K_{X}+\widetilde{C}))\cong H^{1}(\mathbb{F}_{n},\sigma_{*}\Omega^{1}_{X}(\log\widetilde{C})(2K_{X}+\widetilde{C})).
    \end{equation*}
    which gives us the result.
\end{proof}
\begin{lemma}\label{propisomh1XaC}
    If $a\geq 3$ and $b\geq (a-1)n+3$, then 
    $H^{1}(C,\mathcal{O}_{C}(2K_{\mathbb{F}_{n}}+C))\cong H^{0}(\mathbb{F}_{n},\mathcal{O}_{\mathbb{F}_{n}}(-K_{\mathbb{F}_{n}}))$.
\end{lemma}
\begin{proof}
    From the exact sequence
    \begin{equation*}
        0 \rightarrow\mathcal{O}_{\mathbb{F}_{n}}(-C)\rightarrow\mathcal{O}_{\mathbb{F}_{n}}\rightarrow\mathcal{O}_{C}\rightarrow 0,
    \end{equation*}
    by tensoring with $\mathcal{O}_{\mathbb{F}_{n}}(2K_{\mathbb{F}_{n}}+C)$ we obtain 
    \begin{equation*}
        0 \rightarrow\mathcal{O}_{\mathbb{F}_{n}}(2K_{\mathbb{F}_{n}})\rightarrow\mathcal{O}_{\mathbb{F}_{n}}(2K_{\mathbb{F}_{n}}+C)\rightarrow\mathcal{O}_{C}(2K_{\mathbb{F}_{n}}+C)\rightarrow 0
    \end{equation*}
    with exact long sequence in cohomology:
    \begin{equation*}
    \begin{split}
        0&\rightarrow H^{0}(\mathbb{F}_{n},\mathcal{O}_{\mathbb{F}_{n}}(2K_{\mathbb{F}_{n}}))\rightarrow H^{0}(\mathbb{F}_{n},\mathcal{O}_{\mathbb{F}_{n}}(2K_{\mathbb{F}_{n}}+C))\rightarrow H^{0}(C,\mathcal{O}_{C}(2K_{\mathbb{F}_{n}}+C))\rightarrow\\
       &\rightarrow H^{1}(\mathbb{F}_{n},\mathcal{O}_{\mathbb{F}_{n}}(2K_{\mathbb{F}_{n}}))\rightarrow H^{1}(\mathbb{F}_{n},\mathcal{O}_{\mathbb{F}_{n}}(2K_{\mathbb{F}_{n}}+C))\rightarrow H^{1}(C,\mathcal{O}_{C}(2K_{\mathbb{F}_{n}}+C))\rightarrow\\
        &\rightarrow H^{2}(\mathbb{F}_{n},\mathcal{O}_{\mathbb{F}_{n}}(2K_{\mathbb{F}_{n}}))\rightarrow H^{2}(\mathbb{F}_{n},\mathcal{O}_{\mathbb{F}_{n}}(2K_{\mathbb{F}_{n}}+C))\rightarrow 0.\\
    \end{split}
    \end{equation*}
   If $a\geq3$ and $b\geq (a-1)n+3$ then $K_{\mathbb{F}_{n}}+C$ is an ample divisor on $\mathbb{F}_{n}$ then  $H^{i}(\mathbb{F}_{n},\mathcal{O}_{\mathbb{F}_{n}}(2K_{\mathbb{F}_{n}}+C))=0$ for $i=1,2$, by Kawamata–Viehweg vanishing theorem. Therefore, we have $$H^{1}(C,\mathcal{O}_{C}(2K_{\mathbb{F}_{n}}+C))\cong H^{2}(\mathbb{F}_{n},\mathcal{O}_{\mathbb{F}_{n}}(2K_{\mathbb{F}_{n}}))\cong H^{0}(\mathbb{F}_{n},\mathcal{O}_{\mathbb{F}_{n}}(-K_{\mathbb{F}_{n}})).$$
   It is well known that 
   \begin{equation*}
   h^{0}(\mathbb{F}_{n},\mathcal{O}_{\mathbb{F}_{n}}(-K_{\mathbb{F}_{n}})) = \left\{ \begin{array}{lcc} 9 & & n\leq 2,  \\ &  & \\  n+6 &  & n\geq 3. \end{array} \right.
\end{equation*}
\end{proof}
In summary, from all results above, we get the following:
\begin{theorem}\label{corangoh0rhofinal}
    Let $Z=\{p_{1},\dots,p_{\delta}\}$ be $\delta$ distinct reduced closed points on $\mathbb{F}_{n}$. Suppose that $C$ is a nodal curve on $\mathbb{F}_{n}$, with nodes at $Z$, and linearly equivalent to $aC_{0}+bF$ where $a,b\geq 0$. Let $\sigma:X=Bl_{Z}\mathbb{F}_{n}\rightarrow\mathbb{F}_{n}$ be the blow-up of $\mathbb{F}_{Z}$ at $Z$. Let $\widetilde{C}$ be the normalization of $C$ under $\sigma$.\\ Assume that $a\geq \max\{5,\delta+2\}$ and $b\geq\max\{(a-2)n+6,(a-1)n+\delta+2\}$, or that, $a\geq\max\{5,\delta+2\}$ and $b\geq\max\{5,\delta+2\}$ if $n\neq0$. Then 
\begin{equation*}
    \textnormal{corank} H^{0}(\rho)=h^{0}(\mathbb{F}_{n},\mathcal{O}_{\mathbb{F}_{n}}(-K_{\mathbb{F}_{n}})) = \left\{ \begin{array}{lcc} 9 & & n\leq 2,  \\ &  & \\  n+6 &  & n\geq 3. \end{array} \right.
\end{equation*}
\end{theorem}
\begin{proof}
    By Remark~\ref{remarkh0rho1}, we have
$\operatorname{corank} H^{0}(\rho)
    = h^{1}(X,\,\Omega^{1}_{X}(\log \widetilde{C})(2K_{X}+\widetilde{C}).$
Moreover, by Proposition~\ref{propisomlogandh1},
\[
    h^{1}(X,\Omega^{1}_{X}(\log \widetilde{C})(2K_{X}+\widetilde{C}))
    = h^{1}(C,\,\mathcal{O}_{\mathbb{F}_{n}}(2K_{\mathbb{F}_{n}}+C)\!\mid_{C}).
\]
Finally, Proposition~\ref{propisomh1XaC} yields
\[
h^{1}\!\left(C,\mathcal{O}_{C}(2K_{\mathbb{F}_{n}}+C)\right)
    = h^{0}\!\left(\mathbb{F}_{n},\mathcal{O}_{\mathbb{F}_{n}}(-K_{\mathbb{F}_{n}})\right),
\]
which proves the theorem.
\end{proof}
\section{Surjectivity of Gaussian map $\Phi_{X,\mathcal{O}_{X}(K_{X}+\widetilde{C})}$ on the Blow-up surface}
\label{sec:surjectivity}
\subsection{Strategy}
In this section we are going to show the surjectivity of $\Phi_{X,\mathcal{O}_{X}(K_{X}+\widetilde{C})}$. For that we require decompose a certain divisor on $X$ as a sum of three very ample divisor; but with such decomposition necessarily will increase the genus of $\widetilde{C}$.\\
     
     Let $\pi: Y:=Bl_{\Delta}(X \times X ) \rightarrow X\times X$ be the blow-up of $X\times X$ along the diagonal $\Delta\subset X\times X$. Let $\Lambda\subset Y$ be the exceptional divisor. For any coherent sheaf $\mathcal{G}$ on $X$ let us denote by $(\mathcal{G})_{i}=(pr_{i}\circ\pi)^{*}\mathcal{G}$, where $Y\xrightarrow{\pi}X\times X\xrightarrow{pr_{i}}X$, $pr_{i}$ is the $i-$th projection, $i=1,2$. Also, for every divisor $D$ in $X$, we define a divisor in $Y$, as $(D)_{i}=(pr_{i}\circ\pi)^{*}D$. By \cite{MR1201392}, it is well known that a sufficient condition for the surjectivity of $\Phi_{X,\mathcal{O}_{X}(K_{X}+\widetilde{C})}$ is the vanishing of 
     \begin{equation}
         H^{1}(X\times X, pr_{1}^{*}(K_{X}+\widetilde{C})\otimes pr_{2}^{*}(K_{X}+\widetilde{C})\otimes \mathcal{I}_{\Delta}^{2})\cong H^{1}(Y, (K_{X}+\widetilde{C})_{1}+(K_{X}+\widetilde{C})_{2} -2  \Lambda).
     \end{equation}
     Since $K_{Y}\cong \pi^{*}K_{X\times X}+\Lambda$ and $\pi^{*}K_{X\times X}=\pi^{*}(pr_{1}^{*}K_{X}+pr_{2}^{*}K_{X})$ then $K_{Y}\cong(K_{X})_{1}+(K_{X})_{2}+\Lambda$, and 
     \begin{equation*}
         (K_{X}+\widetilde{C})_{1}+(K_{X}+\widetilde{C})_{2} -2  \Lambda= K_{Y}+(\widetilde{C})_{1}+(\widetilde{C})_{2} -3\Lambda.
     \end{equation*}
     Let $\mathcal{L}=\mathcal{O}_{Y}(K_{Y}+(\widetilde{C})_{1}+(\widetilde{C})_{2} -3\Lambda)$, therefore we want to prove that:
     \begin{equation}
         H^{1}(Y,\mathcal{L})=0.
     \end{equation}
\subsection{Very Ample Decompositions}
We define 
\begin{equation*} 
\begin{split}
A & =  \left[\frac{a}{3}\right] C_{0}+ \left[\frac{b}{3}\right] F\\
B & = \left(a- 2 \left[\frac{a}{3}\right]\right) C_{0}+ \left(b-2\left[\frac{b}{3}\right]\right) F\\
M & =\widetilde{C}-E=\sigma^{*}C-3E,
\end{split}
\end{equation*}
where $[r]$ denotes the floor function of a real number $r$.\\
Then 
\begin{equation}\label{M as sum}
    M= 2(\sigma^{*}A-E)+(\sigma^{*}B-E)
\end{equation}
 By the Reider's criterion \cite{MR932299}, the next lemma gives numerical conditions that guarantee the very ampleness of the divisor $\sigma^{*}A - E$ on $X$. For $\delta = 1$, these conditions can be weakened, since $X$ is a toric surface arising as the monoidal transformation of $\mathbb{F}_{n}$, in this case, it suffices to verify ampleness, which follows from the Nakai–Moishezon criterion.
\begin{lemma} \label{lemmaA}
    If $a\geq 6$ and $b\geq \max\{(a+3)n,6\delta-3n+3\}$,
then $\sigma^{*}A-E$ is a very ample divisor on $X$.
\end{lemma}
\begin{proof}
   Let us define $N:=\sigma^{*}A-E-K_{X}$, as $K_{X}=\sigma^{*}K_{\mathbb{F}_{n}}+E$, then 
\begin{equation*}
\begin{split}
N & = \sigma^{*}(A-K_{\mathbb{F}_{n}})-2E \\
 & = \left( \left[\frac{a}{3}\right]+2\right)\sigma^{*}C_{0}+ \left( \left[\frac{b}{3}\right]+n+2\right)\sigma^{*}F-2E
\end{split}
\end{equation*}
We want to prove that $N+K_{X}$ is a very ample divisor, so we will use the Reider's criterion as follow: first, we will see that $N^{2}\geq 10$, then we will show that there is no effective divisor $D$ such that any of the following conditions hold:
 \begin{enumerate}[label=(\roman*)]
     \item $N.D=0$ and $D^{2}=-1$ or $-2$,
      \item $N.D=1$ and $D^{2}=0$ or $-1$,
       \item $N.D=2$ and $D^{2}=0$.
 \end{enumerate}
Note that, 
\begin{equation*}
\begin{split}
N^{2} & = \left(\left[\frac{a}{3}\right]+2\right)^{2}(\sigma^{*}C_{0})^{2}+\left(\left[\frac{b}{3}\right]+n+2\right)^{2}(\sigma^{*}F)^{2}+4E^{2}+ \\
& \hspace{0.3cm}+ 2 \left(\left[\frac{a}{3}\right]+2\right)\left(\left[\frac{b}{3}\right]+n+2\right)(\sigma^{*}C_{0}).(\sigma^{*}F)-4\left(\left[\frac{a}{3}\right]+2\right)(\sigma^{*}C_{0}).E-4\left(\left[\frac{b}{3}\right]+n+2\right)(\sigma^{*}F).E \\
 &=   \left(\left[\frac{a}{3}\right]+2\right)^{2}(-n)+ (-4\delta)+2 \left(\left[\frac{a}{3}\right]+2\right)\left(\left[\frac{b}{3}\right]+n+2\right)\\
& = \left(\left[\frac{a}{3}\right]+2\right)\left(2\left[\frac{b}{3}\right]+2n+4-\left[\frac{a}{3}\right]n-2n\right)-4\delta \\
 & \geq  4\left(2\left[\frac{b}{3}\right]+2n+4-\left[\frac{a}{3}\right]n-2n\right)-4\delta = 8\left[\frac{b}{3}\right]-4\left[\frac{a}{3}\right]n-4\delta+16\geq 10.
\end{split}
\end{equation*}
The last inequality is because $\left[ \frac{a}{3} \right]\geq 2$ and $4\left[\frac{b}{3}\right]\geq 2\left[\frac{a}{3}\right]n+2\delta-3$ by hypothesis.\\

Now, let us consider $\Gamma$ any reduced irreducible curve on $X$ of the form
$$\Gamma= \alpha\sigma^{*}C_{0}+\beta\sigma^{*}F-\sum_{j=1}^{\delta}n_{j}E_{j}$$
then we have the following cases:
\begin{itemize}
    \item If $\alpha=0$ then $\Gamma \in \{\sigma^{*}F, \sigma^{*}F-E_{j}, E_{j}\}$. Thus
 \begin{equation*}
    N.\Gamma = \left\{ \begin{array}{lcc} N.\sigma^{*}F=\left[\frac{a}{3}\right]+2 &  &  \\ N.E_{j}=2 &  & \\ N.(\sigma^{*}F-E_{j})=\left[\frac{a}{3}\right]&  & \end{array} \right.
\end{equation*}
and $N.\Gamma\geq2$ for any $\Gamma\in\{\sigma^{*}F, \sigma^{*}F-E_{j}, E_{j}\}$.
\item We recall that $p_{j}\notin C_{0}$ then, if $\beta=0$ then $\Gamma \in \{\sigma^{*}C_{0}, E_{j}\}$. Thus
 \begin{equation*}
    N.\Gamma = \left\{ \begin{array}{lcc} N.\sigma^{*}C_{0}=\left[\frac{b}{3}\right]-\left[\frac{a}{3}\right]n-n+2 &  & \\ N.E_{j}=2&  & \end{array} \right.
\end{equation*}
and $N.\Gamma\geq 2$ since $\left[\frac{b}{3}\right]\geq\left[\frac{a}{3}\right]n+n$ by hypothesis.
\item If $\alpha\neq 0$ and $\beta\neq 0$, from \cite{MR463157} we have that $\beta\geq\alpha n$. Furthermore, as $\Gamma$ and $\sigma^{*}F-E_{j}$ are distinct irreducible curves, then $\Gamma.(\sigma^{*}F-E_{j})\geq 0$, that is, $\alpha\geq n_{j}$ for all $j=1,\dots,\delta.$
Therefore, 
\begin{equation*}
    \begin{split}
        N.\Gamma & = \left(\left[\frac{a}{3}\right]+2\right)(\beta-\alpha n)+ \alpha \left(\left[\frac{b}{3}\right]+n+2\right)-\sum_{j=1}^{\delta}2n_{j}\,.
    \end{split}
\end{equation*}
By hypothesis, $\left[\frac{b}{3}\right]\geq 2\delta-n+1$ this implies that  $\alpha\left(\left[\frac{b}{3}\right]+n+2\right)\geq \left(\sum_{j=1}^{\delta}2\alpha\right)+3\alpha\geq \left(\sum_{j=1}^{\delta}2\alpha\right)+3$, and $\left(\sum_{j=1}^{\delta}2\alpha\right)\geq \left(\sum_{j=1}^{\delta}2n_{j}\right)$ then $\alpha \left(\left[\frac{b}{3}\right]+n+2\right)-\sum_{j=1}^{\delta}2n_{j}\geq 3$ and, in particular, we obtain that $N.\Gamma\geq 2$.
\end{itemize}
Thus, $N.\Gamma\geq 2$ for any reduced irreducible curve $\Gamma$ on $X$, this tells us that for any effective divisor $D$ on $X$ (that is, a sum of reduced irreducible curves), we have that $N.D\geq 2$. Now, we will focus only on case (iii) $N.D=2$ and $D^{2}=0$.\\
Suppose that there exists an effective divisor $D$ on $X$ such that $N.D=2$ then $D$ is a reduced irreducible curve of the form
$$D= \alpha\sigma^{*}C_{0}+\beta\sigma^{*}F-\sum_{j=1}^{\delta}n_{j}E_{j}.$$
Note that if $\alpha\neq 0$ then $\alpha\geq n_{j}$ for all $j=1,\dots,\delta.$. We consider the following cases:
\begin{itemize}
 \item Case $\alpha\geq1$ and $\beta\geq \alpha n$: \\ By hypothesis,  $\left[\frac{b}{3}\right]\geq 2\delta-n+1$ this implies $\alpha\left(\left[\frac{b}{3}\right]+n+2\right)\geq \left(\sum_{j=1}^{\delta}2\alpha\right)+3\alpha\geq \left(\sum_{j=1}^{\delta}2n_{j}\right)+3$. Then
             \begin{equation*}
                 \begin{split}
                   N.D & = \left(\left[\frac{a}{3}\right]+2\right)(\beta-\alpha n)+ \alpha \left(\left[\frac{b}{3}\right]+n+2\right)-\sum_{j=1}^{\delta}2n_{j}\\
                   &\geq \alpha \left(\left[\frac{b}{3}\right]+n+2\right)-\sum_{j=1}^{\delta}2n_{j}\geq 3 >2.
                 \end{split}
             \end{equation*}
 \item Case $\alpha=1$ and $\beta=0$: \\ In this case, $D$ takes the form $\sigma^{*}C_{0}-\sum_{j=1}^{\delta}n_{j}E_{j}$, where $n_{j}\leq 1$ for all $j=1,\dots,\delta$. Since $D^{2}=0$ then $D^{2}=-n-\sum_{j=1}^{\delta}(n_{j}^{2})=0$, which implies $n=0$ and $n_{j}=0$ for all $j$. Consequently, $N.D=\left[\frac{b}{3}\right]+2>2.$
 \item  Case $\alpha=0$ and $\beta=1$: \\ Here $D=\sigma^{*}F$ or $D=\sigma^{*}F-E_{j_{1}}$ for some $j_{1}\in\{1,\dots,\delta\}$.
       \begin{itemize}
         \item If $D=\sigma^{*}F$, then $N.\sigma^{*}F=\left[\frac{a}{3}\right]+2\geq 4>2$.
         \item If $D=\sigma^{*}F-E_{i}$, then $N.(\sigma^{*}F-E_{i})=\left[\frac{a}{3}\right]\geq 2$ but $(\sigma^{*}F-E_{i})^{2}=E_{i}^{2}=-1\neq 0$.
        \end{itemize}
 \item Case $\alpha=0$ and $\beta=0$:\\ Then $D=E_{j_{1}}$ for some $j_{1}\in \{1,\dots,\delta\}$ and $N.E_{i}=2$, however, $E_{i}^{2}=-1\neq 0$.      
\end{itemize}
Hence, there is no effective divisor $D$ on $X$ such that $N.D=2$ and $D^{2}=0$. Therefore, $\sigma^{*}A-E$ is a very ample divisor on $X$ by Reider's theorem (criterion). 
\end{proof}
Similarly to the proof of Lemma \ref{lemmaA}, we can obtain the following.
\begin{lemma} \label{lemaB}
    If $a\geq 6$ and $b\geq \max\{(a+7)n,6\delta-3n+3\}$,
then $\sigma^{*}B-E$ is a very ample divisor on $X$.
\end{lemma}
Next lemma shows how to construct a big and nef divisor on $Y$ from a very ample divisor on $X$, and its proof is an adaptation of Claim $3.3$ of \cite{MR1092845}.
\begin{lemma} \label{ClaimBEL}
    If $D$ is a very ample divisor on $X$ then $(D)_{1}+(D)_{2}-\Lambda$ is a big and nef divisor on $Y$.
\end{lemma}
\begin{proof}
    Consider the embedding $\phi_{D}:X \rightarrow \mathbb{P}(H^{0}(X,\mathcal{O}_{X}(D))^{*})=\mathbb{P}^{r}$ defined by $D$. We have a rational map $\varphi: X \times X \dashrightarrow G:=Grass(1,\mathbb{P}^{r})$,
    defined on the complement of the diagonal $\Delta$ en $X\times X$ by sending the pair $(x,y)$ to the line $\overline{\phi_{D}(x)\phi_{D}(y)}\subset\mathbb{P}^{r}$, this is called the secant line map.\\
Note that, $\varphi$ extends to a morphism $\overline{\varphi}$
       \begin{equation}
        \begin{split}
            \overline{\varphi}: Y=Bl_{\Delta}(X \times X) \rightarrow G
        \end{split}
    \end{equation}
 defined as follows: If $(x,y)\notin \Delta$, then $\overline{\varphi}(x,y)=\varphi(x,y)=\overline{\phi_{D}(x)\phi_{D}(y)}$. On $\Delta$, $\overline{\varphi}$ maps to the tangent line at $\phi_{D}(x)$ in the direction defined by the exceptional divisor.
 We can check that 
 \begin{equation}
     \begin{split}
         (D)_{1}+(D)_{2}-\Lambda &=\overline{\varphi}^{*}\mathcal{O}_{G}(1),
     \end{split}
 \end{equation}
 where $\mathcal{O}_{G}(1)$ is the positive generator of Pic($G$).\\
Now, unless $X$ is a linear subspace of $\mathbb{P}^{r}$, the map $\overline{\varphi}$ is generically finite; the fiber over a general point $l=\overline{xy}$ in the image 
will be positive-dimensional if and only if $l\subset X$ and the only variety that contains 
the line joining any two of its points is a linear subspace of $\mathbb{P}^{r}$.\\
Therefore $ (D)_{1}+(D)_{2}-\Lambda$ is big and nef divisor because it is the pullbcak of a big and nef divisor under a generically finite morphism.
\end{proof}

Note that the conditions of Lemma \ref{lemaB} satisfy the conditions of Lemma \ref{lemmaA}. We can apply Lemma \ref{ClaimBEL} and obtain the next theorem.
\begin{theorem} \label{theoremBigandNef}
      If $a\geq 6$ and $b\geq \max\{(a+7)n,6\delta-3n+3\}$,
then $(M)_{1}+(M)_{2}-3\Lambda$ is a big and nef divisor on $Y$.
\end{theorem}
\begin{proof}
By Lemmas \ref{lemmaA}, \ref{lemaB} and \ref{ClaimBEL}, both $(\sigma^{*}A-E)_{1}+(\sigma^{*}A-E)_{2}- \Lambda$ and $(\sigma^{*}B-E)_{1}+(\sigma^{*}B-E)_{2}- \Lambda$ are big and nef divisor on $Y$. By (\ref{M as sum}), 
 \begin{equation*}
 \begin{split}
       (M)_{1}+(M)_{2}-3\Lambda & = (\sigma^{*}(2A-B)-3E)_{1}+(\sigma^{*}(2A-B)-3E)_{2}-3\Lambda \\
     & =2\left((\sigma^{*}A-E)_{1}+(\sigma^{*}A-E)_{2}- \Lambda\right)+(\sigma^{*}B-E)_{1}+(\sigma^{*}B-E)_{2}- \Lambda
 \end{split}
 \end{equation*}
 is big and nef divisor on $Y$ since it is the sum of three big and nef divisors.
\end{proof}
 By Kawamata–Viehweg vanishing theorem, we obtain the following corollary.
\begin{corollary}\label{CoroVanishing}
    Under the conditions of Theorem \ref{theoremBigandNef}, 
$$H^{1}(Y, K_{Y}+(M)_{1}+(M)_{2}-3\Lambda)=0.$$
\end{corollary}

Now, consider $Y_{j}\subset Y$ the proper transform of $E_{j}\times X$, and consider the divisors $R_{j}=\widetilde{C}-\sum_{i=1}^{j}E_{i}$ for each $j=1,\dots,\delta$.\\ Note that 
\begin{equation}
    \begin{split}
        R_{j} & = \sigma^{*}C-2E-\sum_{i=1}^{j}E_{i}= \sigma^{*}(2A+B)-2E-\sum_{i=1}^{j}E_{i}\\
        &=2\left(\sigma^{*}A-E\right)+ (\sigma^{*}B-\sum_{i=1}^{j}E_{i})\, .
    \end{split}
\end{equation}
Analogously to the proof of Lemma \ref{lemmaA}, we obtain the following.
\begin{lemma}\label{lemaB_i} If $a\geq 6$ and $b\geq \max\{(a+7)n,6j-3n+3\}$, then the divisors $\sigma^{*}B-\sum_{i=1}^{j}E_{i}$ are very ample divisors on $X$, for all $j=1,\dots,\delta$.
\end{lemma}
\begin{remark}\label{remark1}
 Therefore, if $a\geq 6$ and $b\geq \max\{(a+7)n,6\delta-3n+3\}$,
then $\left(\sigma^{*}A-E\right)$ and $ (\sigma^{*}B-\sum_{i=1}^{j}E_{i})$ are very ample divisors $X$ for all $j=1,\dots,\delta$.
\end{remark}
By Lemma \ref{ClaimBEL}, we obtain, for each $j=1,\dots,\delta$, that the divisors $(\sigma^{*}B-\sum_{i=1}^{j}E_{i})_{1}+(\sigma^{*}B-\sum_{i=1}^{j}E_{i})_{2}-\Lambda$ and
$\left(\sigma^{*}A-E\right)_{1}+ \left(\sigma^{*}A-E\right)_{2}-\Lambda$ are big and nef divisors on $Y$.\\
Thus, for each $j=1,\dots,\delta$, the restrictions of $(\sigma^{*}B-\sum_{i=1}^{j}E_{i})_{1}+(\sigma^{*}B-\sum_{i=1}^{j}E_{i})_{2}-\Lambda$ and $\left(\sigma^{*}A-E\right)_{1}+ \left(\sigma^{*}A-E\right)_{2}-\Lambda$ on $Y_{j}$ are nef divisors. Now, we will see that these restrictions are also big. \\
Let us recall from \cite{MR1760876} that if $D$ is very ample on X, then the linear system $|(D)_{1}+(D)_{2}-\Lambda|$ on $Y$ has a linear subspace that defines a morphism $Y \rightarrow Grass(1,\mathbb{P}H^{0}(X,\mathcal{O}_{X}(D))^{*})$, which associates the pair $(x,y)\in Y$ with the linear space generated by $\phi_{D}(x)$ and $\phi_{D}(y)$, where $\phi_{D}$ is the embedding of $X$ on $\mathbb{P}H^{0}(X,\mathcal{O}_{X}(D))^{*}$. \\
Since $\sigma^{*}A-E$ embeds $X$ so that $E_{j}$ is a line, if we assume that $(\sigma^{*}A-E)_{1}+(\sigma^{*}A-E)_{2}-\Lambda$ were not big on $Y_{j}$, then the restriction to $Y_{j}$ of the map $Y\rightarrow Grass(1,\mathbb{P}H^{0}(X,\mathcal{O}_{X}(D))^{*})$ would have no finite fiber, hence every chord joining a point of the line $\phi_{\sigma^{*}A-E}(E_{j})$ and a point of the image $\phi_{\sigma^{*}A-E}(X)$ would be contained in $\phi_{\sigma^{*}A-E}(X)$, that is $\phi_{\sigma^{*}A-E}(X)$ is a plane, this is a contradiction.\\ Similarly, for each $j=1,\dots,\delta$, the divisors $[(\sigma^{*}B-\sum_{i=1}^{j}E_{i})_{1}+(\sigma^{*}B-\sum_{i=1}^{j}E_{i})_{2}-\Lambda]|_{Y_{j}}$ are big and nef divisors on $Y_{j}$.\\

Applying Lemma \ref{lemaB_i} and Remark \ref{remark1}, we have the next corollary.
\begin{corollary} \label{corolarioban}
    If $a\geq 6$ and $b\geq \max\{(a+7)n,6\delta-3n+3\}$,
then for every $j=1,\dots,\delta$
\begin{equation*}
    \begin{split}
        (\widetilde{C}-\sum_{i=1}^{j}E_{i})_{1}+(\widetilde{C}-\sum_{i=1}^{j}E_{i})_{2}-3\Lambda
    \end{split}
\end{equation*}
is a big and nef divisor on $Y$. In addition, the restriction to $Y_{j}$,
\begin{equation*}
    \begin{split}
       [(\widetilde{C}-\sum_{i=1}^{j}E_{i})_{1}+(\widetilde{C}-\sum_{i=1}^{j}E_{i})_{2}-3\Lambda]|_{Y_{j}}
    \end{split}
\end{equation*}
is big and nef, for all $j=1,\dots,\delta$.\\
\end{corollary}
Now, we will prove that these numerical conditions are sufficient to have the surjectivity of $\Phi_{K_{X}+\widetilde{C}}$.
\subsection{Main Surjectivity Result}

\begin{theorem}\label{h1L=0}
As before, let $\mathcal{L}=K_{Y}+(\widetilde{C})_{1}+(\widetilde{C})_{2} -3\Lambda$. If the conditions of the Corollary \ref{corolarioban} are satisfied, then
    $$H^{1}(Y,\mathcal{O}_{Y}(\mathcal{L}))=0.$$
\end{theorem}
\begin{proof}
    Consider the divisors 
    \begin{itemize}
        \item $Z_{j}\subset Y$ defined by the strict transforms of $X\times E_{j}$
        \item $F_{j}:=(E_{j})_{1}+(E_{j})_{2}=Z_{j}+Y_{j}$, with $Y_{j}$ as before.
        \item $F:=\sum_{j=1}^{\delta}F_{j}$
    \end{itemize}
    We recall that $M= \widetilde{C}-E$ then $(M)_{i}= ( \widetilde{C}-\sum_{j=1}^{\delta}E_{j})_{i}$ for $i=1,2$. Since $\mathcal{L}=K_{Y}+(\widetilde{C})_{1}+(\widetilde{C})_{2} -3\Lambda$, then 
    \begin{equation}
        \begin{split}
            \mathcal{L}-F&= K_{Y}+(\widetilde{C})_{1}+(\widetilde{C})_{2} -3\Lambda-F\\
            &=K_{Y}+(\widetilde{C})_{1}+(\widetilde{C})_{2} -3\Lambda- \left(\sum_{j=1}^{\delta} (E_{j})_{1}+(E_{j})_{2}\right)\\
            &=K_{Y}+(M)_{1}+(M)_{2} -3\Lambda.
        \end{split}
    \end{equation}
    Therefore, by the Corollary \ref{CoroVanishing}, we have that $H^{1}(Y,\mathcal{L}-F)=0$.\\

By convention, we take $F_{0}=0$. Note that, we have a exact sequence for each $i=1,\dots,\delta$
\begin{equation}
    \begin{split}
        0\rightarrow\mathcal{O}_{Y}(-F_{i})\rightarrow\mathcal{O}_{Y}\rightarrow \mathcal{O}_{F_{i}}\rightarrow 0.
    \end{split}
\end{equation}
Tenzoring with $\mathcal{O}_{Y}(\mathcal{L}-\sum_{j=0}^{i-1}F_{j})$ we obtain:
\begin{equation} \label{exactsequenceY}
    \begin{split}
        0\rightarrow\mathcal{O}_{Y}(\mathcal{L}-\sum_{j=0}^{i}F_{j})\rightarrow\mathcal{O}_{Y}(\mathcal{L}-\sum_{j=0}^{i-1}F_{j})\rightarrow \mathcal{O}_{F_{i}}(\mathcal{L}-\sum_{j=0}^{i-1}F_{j})\rightarrow 0
    \end{split}
\end{equation}
For $i=1$, we have the sequence
\begin{equation*}
    \begin{split}
        0\rightarrow\mathcal{O}_{Y}(\mathcal{L}-F_{1})\rightarrow\mathcal{O}_{Y}(\mathcal{L})\rightarrow \mathcal{O}_{F_{1}}(\mathcal{L})\rightarrow 0,
    \end{split}
\end{equation*}
thus, if $H^{1}(Y,\mathcal{O}_{Y}(\mathcal{L}-F_{1}))=0$ and $H^{1}(F_{1},\mathcal{O}_{F_{1}}(\mathcal{L}))=0$,
then we are done. From the exact sequence in (\ref{exactsequenceY}) for $i=2$, if $H^{1}(Y,\mathcal{O}_{Y}(\mathcal{L}-F_{1}-F_{2}))=0$ and $H^{1}(F_{2},\mathcal{O}_{F_{2}}(\mathcal{L}-F_{1}))=0$ then we have the vanishing of $H^{1}(Y,\mathcal{O}_{Y}(\mathcal{L}-F_{1}))$. Recursively until $i=\delta$,  if $H^{1}(Y, \mathcal{O}_{Y}(\mathcal{L}-F))=0$ and $H^{1}(F_{\delta},\mathcal{O}_{F_{\delta}}(\mathcal{L}-\sum_{j=0}^{\delta-1}F_{j}))=0$ we have the vanishing of $H^{1}(Y,\mathcal{O}_{Y}(\mathcal{L}-\sum_{i=0}^{\delta-1}F_{i}))$.\\ Therefore, we only need to show that 
\begin{equation}\label{vanishingH1F_{i}}
    H^{1}(F_{i},\mathcal{O}_{F_{i}}(\mathcal{L}-\sum_{j=0}^{i-1}F_{j}))=0, \text{  for each  } i=1,\dots,\delta.
\end{equation}
Let us denote by $U_{ij}$ the strict transform on $Y$ of $E_{i} \times E_{j}$. Note that if $i \neq j$, then $E_i \cap E_j = \emptyset$, and so
$U_{ij} \cong E_i \times E_j.$\\
By definition, the strict transforms of the divisors $E_i \times X$ and $X \times E_i$ under the blow-up $\pi: Y \to X \times X$ are $Y_i = \pi^*(E_i \times X) - \Lambda|_{Y_i}$ and $Z_i = \pi^*(X \times E_i) - \Lambda|_{Z_i}$ where $\Lambda$ is the exceptional divisor.\\
The subscheme $E_i \times E_i = (E_i \times X) \cap (X \times E_i) \subset X \times X$ intersects the diagonal $\Delta$ along
$\Delta_{E_i} \cong E_i.$\\
The strict transform $U_{ii}$ of $E_i \times E_i$ is, by definition, $U_{ii} = \text{strict transform of } E_i \times E_i
       = \pi^*(E_i \times E_i) - \Lambda|_{E_i}.$\\
On the other hand, the intersection of the strict transforms
\[
W_i := Y_i \cap Z_i
\]
contains both the strict transform $U_{ii}$ and the contribution from the exceptional divisor restricted to $Y_i$. Therefore, as divisors on $Y_i$, one has
\[
W_i = U_{ii} + \Lambda|_{Y_i}.
\]

By the notation introduced above, we have $F_i = Y_i \cup Z_i$ and $W_i = Y_i \cap Z_i$. Then, as divisors on $Y_i$, $W_i$ and $U_{ii} + \Lambda|_{Y_i}$ are linearly equivalent. Consequently, there is a short exact sequence of sheaves
\begin{equation}
0 \longrightarrow \mathcal{O}_{Y_i}(-W_i) \longrightarrow \mathcal{O}_Y \longrightarrow \mathcal{O}_{Z_i} \longrightarrow 0.
\end{equation}
Tenzoring with $\mathcal{O}_{F_{i}}(\mathcal{L}-\sum_{j=0}^{i-1}F_{j})$ we obtain
\begin{equation}\label{exactsequenceY_{i}}
    0\rightarrow \mathcal{O}_{Y_{i}}(\mathcal{L}-W_{i}-\sum_{j=0}^{i-1}F_{j})\rightarrow\mathcal{O}_{F_{i}}(\mathcal{L}-\sum_{j=0}^{i-1}F_{j})\rightarrow\mathcal{O}_{Z_{i}}(\mathcal{L}-\sum_{j=0}^{i-1}F_{j})\rightarrow 0.
\end{equation}
On $Y_{i}$ also we have a exact sequence
\begin{equation}
    0\rightarrow \mathcal{O}_{Y_{i}}(\mathcal{L}-W_{i}-\sum_{j=0}^{i-1}F_{j})\rightarrow\mathcal{O}_{Y_{i}}(\mathcal{L}-\sum_{j=0}^{i-1}F_{j})\rightarrow\mathcal{O}_{W_{i}}(\mathcal{L}-\sum_{j=0}^{i-1}F_{j})\rightarrow 0.
\end{equation}
since $W_{i}\cap F_{j}=\emptyset$ for $j\leq i-1$, this sequence is reduced to
\begin{equation}\label{exactsequenceY_{i}2}
    0\rightarrow \mathcal{O}_{Y_{i}}(\mathcal{L}-W_{i}-\sum_{j=0}^{i-1}F_{j})\rightarrow\mathcal{O}_{Y_{i}}(\mathcal{L}-\sum_{j=0}^{i-1}F_{j})\rightarrow\mathcal{O}_{W_{i}}(\mathcal{L})\rightarrow 0.
\end{equation}
From the sequence (\ref{exactsequenceY_{i}}), if for each $i=1\dots,\delta$
\begin{equation}
    \begin{split}
        H^{1}(\mathcal{O}_{Y_{i}}(\mathcal{L}-W_{i}-\sum_{j=0}^{i-1}F_{j}))&=0\\
        H^{1}(\mathcal{O}_{Z_{i}}(\mathcal{L}-\sum_{j=0}^{i-1}F_{j}))&=0
    \end{split}
\end{equation}
then (\ref{vanishingH1F_{i}}) is satisfied. However, due to the definition of $Z_{i}$ and the symmetry with $Y_{i}$, by the exact sequence in (\ref{exactsequenceY_{i}2}), it reduces to proving that, for each $i=1,\dots,\delta$
\begin{equation}\label{parte1}
        H^{1}(\mathcal{O}_{Y_{i}}(\mathcal{L}-W_{i}-\sum_{j=0}^{i-1}F_{j}))=0
\end{equation}
\begin{equation}\label{parte2}
        H^{1}(\mathcal{O}_{W_{i}}(\mathcal{L})=0
\end{equation}
To prove (\ref{parte1}), let us observe that, $Y_{i}=(E_{i})_{1}$ and $Z_{i}=(E_{i})_{2}$. Therefore, as divisors on $Y_{i}$, $W_{i}=Z_{i}|_{Y_{i}}=[(E_{i})_{2}]|_{Y_{i}}$ then
\begin{equation*}
    \begin{split}
        \mathcal{L}|_{Y_{i}}=[K_{Y}+(\widetilde{C})_{1}+(\widetilde{C})_{2}-3\Lambda]|_{Y_{i}}&=[K_{Y}+(E_{i})_{1}+(E_{i})_{2}+(\widetilde{C}-E_{i})_{1}+(\widetilde{C}-E_{i})_{2}-3\Lambda]|_{Y_{i}}\\
        &\cong [K_{Y}+(E_{i})_{1}]|_{(E_{i})_{1}}+[(E_{i})_{2}]|_{Y_{i}}+[(\widetilde{C}-E_{i})_{1}+(\widetilde{C}-E_{i})_{2}-3\Lambda]|_{Y_{i}}\\
        & \cong K_{Y_{i}}+ W_{i}+ [(\widetilde{C}-E_{i})_{1}+(\widetilde{C}-E_{i})_{2}-3\Lambda]|_{Y_{i}}
    \end{split}
\end{equation*}
then
\begin{equation}
    \mathcal{O}_{Y_{i}}(\mathcal{L}-W_{i}-\sum_{j=0}^{i-1}F_{j})\cong \mathcal{O}_{Y_{i}}(K_{Y_{i}} +  (\widetilde{C}-E_{i})_{1}+(\widetilde{C}-E_{i})_{2}-3\Lambda -\sum_{j=0}^{i-1}F_{j})
\end{equation}
note that
\begin{equation*}
    \begin{split}
        (\widetilde{C}-E_{i})_{1}+(\widetilde{C}-E_{i})_{2} -\sum_{j=0}^{i-1}F_{j}&=(\widetilde{C})_{1}+(\widetilde{C})_{2}-F_{i}-\sum_{j=0}^{i-1}F_{j}\\
        &= (\widetilde{C})_{1}+(\widetilde{C})_{2}-\sum_{j=0}^{i}F_{j}=(\widetilde{C})_{1}+(\widetilde{C})_{2}-\sum_{j=1}^{i}(E_{j})_{1}+(E_{j})_{2}\\
        &=(\widetilde{C}-\sum_{j=1}^{i}E_{j})_{1}+(\widetilde{C}-\sum_{j=1}^{i}E_{i})_{2}
    \end{split}
\end{equation*}
Therefore, $\mathcal{O}_{Y_{i}}(\mathcal{L}-W_{i}-\sum_{j=0}^{i-1}F_{j})\cong \mathcal{O}_{Y_{i}}(K_{Y_{i}} +  (\widetilde{C}-\sum_{j=1}^{i}E_{j})_{1}+(\widetilde{C}-\sum_{j=1}^{i}E_{j})_{2}-3\Lambda )$ and by Corollary \ref{corolarioban} we have that $[(\widetilde{C}-\sum_{i=1}^{j}E_{i})_{1}+(\widetilde{C}-\sum_{i=1}^{j}E_{i})_{2}-3\Lambda]|_{Y_{j}}$ is a big and nef divisor on $Y_{j}$ and by Kawamata Viehweg vanishing theorem we obtain (\ref{parte1}).\\

To prove (\ref{parte2}), we will proceed as in (\cite{MR1760876}, Lemma 3.1), by selfcontained we give the proof of such fact.\\ Let $\widetilde{\Lambda}_{i}=\Lambda|_{Y_{i}}$ so $\widetilde{\Lambda}_{i}\cong \mathbb{P}\mathcal{E}$ with $\mathcal{E}\cong N^{*}_{\Delta_{E_{i}}/E_{i}\times X}$. From the exact sequence of normal bundle
\begin{equation}
    0\rightarrow N_{\Delta_{E_{i}}/E_{i}\times E_{i}}\rightarrow N_{\Delta_{E_{i}}/E_{i}\times X} \rightarrow N_{E_{i}\times E_{i}/E_{i}\times X}|_{\Delta_{E_{i}}}\rightarrow 0
\end{equation}
and the isomorphism $N_{\Delta_{E_{i}}/E_{i}\times E_{i}}\cong T_{E_{i}}\cong\mathcal{O}_{\mathbb{P}^{1}}(2)$ and $N_{E_{i}\times E_{i}/E_{i}\times X}|_{\Delta_{E_{i}}}\cong N_{E_{i}/X}\cong \mathcal{O}_{\mathbb{P}^{1}}(-1)$, so $N_{\Delta_{E_{i}}/E_{i}\times X} \cong \mathcal{O}_{\mathbb{P}^{1}}(2)\oplus\mathcal{O}_{\mathbb{P}^{1}}(-1)$ and $\mathcal{E}\cong\mathcal{O}_{\mathbb{P}^{1}}(-2)\oplus\mathcal{O}_{\mathbb{P}^{1}}(1)$, thus $\mathbb{P}\mathcal{E}\cong\mathbb{P}(\mathcal{O}_{\mathbb{P}^{1}}(-3)\oplus\mathcal{O}_{\mathbb{P}^{1}})\cong \mathbb{F}_{3}$.\\ 
As
\[
\mathcal{L} = \mathcal{O}_{Y}\big((K_X + \widetilde{C})_1 + (K_X + \widetilde{C})_2 - 2\Lambda\big)
\]
then restricting \(\mathcal{L}\) to \(\widetilde{\Lambda}_i\) gives
\[
\mathcal{L}\big|_{\widetilde{\Lambda}_i} = -2 \Lambda\big|_{\widetilde{\Lambda}_i}.
\]
Under the identification \(\widetilde{\Lambda}_i \cong \mathbb{P}\mathcal{E}\), the tautological class \(\mathcal{O}_{\widetilde{\Lambda}_i}(1)\) corresponds to the divisor \(C_0 + F\), where \(C_0 \in |\mathcal{O}_{\mathbb{P}\mathcal{E}}(1)|\) and \(F\) is a fiber of the projection \(\mathbb{P}\mathcal{E} \to E_i \cong \mathbb{P}^1\). Hence, we have
$
\Lambda\big|_{\widetilde{\Lambda}_i} = \mathcal{O}_{\mathbb{P}\mathcal{E}}(-C_0 - F),
$
and therefore
\[
\mathcal{L}\big|_{\widetilde{\Lambda}_i} \cong \mathcal{O}_{\mathbb{P}\mathcal{E}}(2C_0 + 2F).
\]
Furthermore, the intersection $Q_{i}=U_{ii}\cap \widetilde{\Lambda}$ is isomorphic to $\Delta_{E_{i}}$ and is a divisor of the type $C_{0}+\alpha F$ in $\mathbb{P}\mathcal{E}$ for some $\alpha$. Therefore
\begin{equation*}
    h^{1}(\mathcal{O}_{\widetilde{\Lambda}_{i}}(\mathcal{L}))=h^{1}(\mathcal{O}_{\mathbb{F}_{3}}(2C_{0}+2F))=1.
\end{equation*}
\begin{equation*}
    h^{2}(\mathcal{O}_{\widetilde{\Lambda}_{i}}(\mathcal{L}-Q_{i}))=h^{2}(\mathcal{O}_{\mathbb{F}_{3}}(C_{0}+(2-\alpha)F))=0.
\end{equation*}
From the exact sequence
\begin{equation*}
    0\rightarrow \mathcal{O}_{\widetilde{\Lambda}_{i}}(-\widetilde{\Lambda}_{i}\cap U_{ii})\rightarrow \mathcal{O}_{W_{i}}\rightarrow\mathcal{O}_{U_{ii}}\rightarrow0
\end{equation*}
tensoring with $\mathcal{O}_{Y}(\mathcal{L})$ we obtain
\begin{equation}
    0\rightarrow \mathcal{O}_{\widetilde{\Lambda}_{i}}(\mathcal{L}-Q_{i})\rightarrow \mathcal{O}_{W_{i}}(\mathcal{L})\rightarrow\mathcal{O}_{U_{ii}}(\mathcal{L})\rightarrow0
\end{equation}
since $\mathcal{O}_{U_{ii}}(\mathcal{L})\cong\mathcal{O}_{\mathbb{P}^{1}\times\mathbb{P}^{1}}(-1,-1)$ then $H^{2}(\mathcal{O}_{U_{ii}}(\mathcal{L}))=0$ and therefore $H^{2}(\mathcal{O}_{W_{i}}(\mathcal{L}))=0$.\\
Now, from the exact sequence 
\begin{equation}
    0\rightarrow \mathcal{O}_{U_{ii}} \rightarrow \mathcal{O}_{W_{i}}\rightarrow \mathcal{O}_{\widetilde{\Lambda}_{i}}\rightarrow 0
\end{equation}
tensoring with $\mathcal{O}_{Y}(\mathcal{L})$ we obtain
\begin{equation}
    0\rightarrow \mathcal{O}_{U_{ii}}(\mathcal{L}-Q_{i})\rightarrow \mathcal{O}_{W_{i}}(\mathcal{L})\rightarrow\mathcal{O}_{\widetilde{\Lambda}_{i}}(\mathcal{L})\rightarrow0
\end{equation}
since $\mathcal{O}_{U_{ii}}(\mathcal{L}-Q_{i})\cong \mathcal{O}_{\mathbb{P}^{1}\times\mathbb{P}^{1}}(-2,-2)$ then $h^{1}(\mathcal{O}_{U_{ii}}(\mathcal{L}-Q_{i}))=0$ and $h^{2}(\mathcal{O}_{U_{ii}}(\mathcal{L}-Q_{i}))=1$. And finally we get that $H^{1}(\mathcal{O}_{W_{i}}(\mathcal{L}))=0$.
\end{proof}
Hence, the morphism $\Phi_{K_{X}+\widetilde{C}}$ is surjective. The summary is as follows.
\begin{theorem}\label{teoremafinaldeltanodos}
    Let $p_{1},\dots,p_{\delta}$ be $\delta$ distinct points on $\mathbb{F}_{n}$ such that for each $j$, $p_{j}\notin C_{0}$, and for $i,j\in\{1,\dots,\delta\}$, $p_{i}$ and $p_{j}$ are not on the same fiber of the fibration that admits $\mathbb{F}_{n}$, $\phi:\mathbb{F}_{n}\rightarrow\mathbb{P}^{1}$.\\
Suppose that $C$ is a nodal curve on $\mathbb{F}_{n}$ in the linear system $|D|$, with $D$ linearly equivalent to $aC_{0}+bF$.
Let $\sigma:X=Bl_{Z}\mathbb{F}_{n}\rightarrow\mathbb{F}_{n}$ be the blow-up of $\mathbb{F}_{n}$ at $Z$, where $Z=\{p_{1},\dots,p_{\delta}\}$ are the nodes of $C$. Let $\widetilde{C}$ be the normalization of $C$ under $\sigma$.\\ Assume that $a\geq 6$ and $b\geq \max\{(a+7)n,(a-1)n+\delta+2,6\delta-3n+3\}$,
then:
\begin{enumerate}
    \item  the Gaussian map $\Phi_{X,\mathcal{O}_{X}(K_{X}+\widetilde{C})}$ 
    is surjective.
    \item $\textnormal{corank}(\Phi_{\widetilde{C}})=h^{0}(\mathbb{F}_{n},\mathcal{O}_{\mathbb{F}_{n}}(-K_{\mathbb{F}_{n}}))$.
\end{enumerate}
\end{theorem}
\begin{proof}
    For (1), the conditions imply those of Theorem \ref{h1L=0}. For (2), combine Remark \ref{remarkcorank2} and Theorem \ref{corangoh0rhofinal}.\\
\end{proof}
As mentioned above, for the case $\delta=1$ we can weaken numerical conditions for the surjectivity of $\Phi_{X,\mathcal{O}_{X}(K_{X}+\widetilde{C})}$. Instead of using Lemma \ref{lemmaA} and Lemma \ref{lemaB}, we will use the next lemma.
\begin{lemma}\label{lema1nodoveryample}
$\sigma^{*}A-E$ and $\sigma^{*}B-E$ are very ample divisors on $X$ if
    \begin{align}
        \left[\frac{b}{3}\right] & \geq \left[\frac{a}{3}\right]n+1,\label{condicionb2}\\
        \left[\frac{a}{3}\right] & \geq 2. \label{condiciona2}
    \end{align}
\end{lemma}
\begin{proof}
We note that $\mathbb{F}_{n}$ a toric surface, thus every ample divisor is a very ample divisor. In particular, $X$ is a toric surface since is a monoidal transformation of $\mathbb{F}_{n}$ \cite{MR2810322}. Therefore, it is enough to prove that $\sigma^{*}A-E$ and $\sigma^{*}B-E$ are ample divisor and we will use the Nakai-Moishezon criterion.\\
    First, we can see 
    \begin{equation*}
        (\sigma^{*}A-E)^{2}=A^{2}-1=-\left[\frac{a}{3}\right]^{2}n+2\left[\frac{a}{3}\right]\left[\frac{b}{3}\right]-1 = \left[\frac{a}{3}\right]\left( \left[\frac{b}{3}\right]-\left[\frac{a}{3}\right]n\right)-1 >0
    \end{equation*}
    with the conditions 
    \begin{align}
        \left[\frac{b}{3}\right] & \geq \left[\frac{a}{3}\right]n+1, \label{condicionb}\\
        \left[\frac{a}{3}\right] & \geq 2. \label{condiciona}
    \end{align}
Now, for any irreducible curve $\Gamma\subset X$ we have to show that $(\sigma^{*}A-E).\Gamma>0$.\\
Since $\Gamma$ has the form $\sigma^{*}G-\gamma E$, where $G=\alpha_{1}C_{0}+\alpha_{2}F$ and $\gamma\in \mathbb{Z}$.\\
If $\gamma<0$ then $(\sigma^{*}A-E).\Gamma=A.G-\gamma >0$ because $A$ is very ample divisor with the above conditions. Let us consider the case $\gamma\geq 0$.\\
If $\alpha_{1}=0$ then $G=F$ and
\begin{equation*}
    \Gamma = \left\{ \begin{array}{lcc} \sigma^{*}F-E & & p\in F \\ \\  \sigma^{*}F &  & p\notin F \end{array} \right.
\end{equation*}
so, 
\begin{equation*}
    (\sigma^{*}A-E).\Gamma = \left\{ \begin{array}{lcc} \left[\frac{a}{3}\right] & & p\in F \\ \\\left[\frac{a}{3}\right]-1 &  & p\notin F \end{array} \right. 
\end{equation*}
is positive in any case, if  $(\ref{condiciona})$ is satisfied. \\
If $\alpha_{2}=0$ then $G=C_{0}$ and
\begin{equation*}
    \Gamma = \sigma^{*}C_{0}
\end{equation*}
Since the point $p$ does not belong to $C_{0}$ then the divisor $E$ does not appear in $\Gamma$. It follows that
\begin{equation*}
    (\sigma^{*}A-E).\Gamma =-\left[\frac{a}{3}\right]n+\left[\frac{b}{3}\right] >0.
\end{equation*}
Assume that $\alpha_{1}\neq 0$, $\alpha_{2}\neq 0$. \\
Then $\Gamma$ and $\sigma^{*}F-E$ are distinct curves, $\Gamma.(\sigma^{*}F-E)\geq0$ that is $\alpha_{1}-\gamma\geq 0$.\\
If $n=0$ then $\alpha_{1}>0$ and $\alpha_{2}>0$, thus
\begin{equation*}
    (\sigma^{*}A-E).\Gamma=\left[\frac{a}{3}\right] \alpha_{2}+\left[\frac{b}{3}\right]\alpha_{1}-\gamma > \left[\frac{b}{3}\right]\alpha_{1}-\gamma \geq 0.
\end{equation*}
In the case that $n>0$, then $\alpha_{1}>0$, $\alpha_{2}\geq \alpha_{1} n$. In this case
\begin{equation*}
    (\sigma^{*}A-E).\Gamma=\left[\frac{a}{3}\right] \left(\alpha_{2}-\alpha_{1}n\right)+\left[\frac{b}{3}\right]\alpha_{1}-\gamma \geq \left[\frac{b}{3}\right]\alpha_{1}-\gamma > 0
\end{equation*}
if $\left[\frac{b}{3}\right] \geq 2$ which is satisfied by $(\ref{condicionb})$.\\
With the conditions $(\ref{condicionb})$ and $(\ref{condiciona})$, we have that 
 \begin{equation*}
     \left(\sigma^{*}B-E\right)^{2}=B^{2}-1=\left(a-2\left[\frac{a}{3}\right] \right) \left(2\left(  b-2\left[\frac{b}{3}\right]\right)-\left(a-2\left[\frac{a}{3}\right]\right)n \right)-1 > 0.
 \end{equation*}   
Now, for any irreducible curve $\Gamma\subset X$ we have to show that $(\sigma^{*}B-E).\Gamma>0$.\\
Since $\Gamma$ has the form $\sigma^{*}G-\gamma E$, where $G=\beta_{1}C_{0}+\beta_{2}F$ and $\gamma\in \mathbb{Z}$.\\
If $\gamma<0$ then $(\sigma^{*}B-E).\Gamma=A.G-\gamma >0$ because $B$ is very ample with the above conditions. Let us consider the case $\gamma\geq 0$.\\
If $\beta_{1}=0$ then $G=F$ and
\begin{equation*}
    \Gamma = \left\{ \begin{array}{lcc} \sigma^{*}F-E & & p\in F \\ \\  \sigma^{*}F &  & p\notin F \end{array} \right.
\end{equation*}
so, 
\begin{equation*}
    (\sigma^{*}B-E).\Gamma = \left\{ \begin{array}{lcc} a-2\left[\frac{a}{3}\right] & & p\in F \\ \\ a-2\left[\frac{a}{3}\right]-1 &  & p\notin F \end{array} \right. 
\end{equation*}
is positive in any case, by  $(\ref{condiciona})$. \\
If $\beta_{2}=0$ then $G=C_{0}$ and
\begin{equation*}
    \Gamma = \sigma^{*}C_{0}
\end{equation*}
Since the point $p$ does not belong to $C_{0}$ then the divisor $E$ does not appear in $\Gamma$. It follows that
\begin{equation*}
    (\sigma^{*}B-E).\Gamma =-\left( a-2\left[\frac{a}{3}\right]\right)n+\left( b-2\left[\frac{b}{3}\right]\right) >0.
\end{equation*}
Assume that $\beta_{1}\neq 0$, $\beta_{2}\neq 0$. \\
Then $\Gamma$ and $\sigma^{*}F-E$ are distinct curves, then $\Gamma.(\sigma^{*}F-E)\geq 0$ that is $\beta_{1}-\gamma\geq 0$.\\
If $n=0$ then $\beta_{1}>0$ and $\beta_{2}>0$, thus
\begin{equation*}
    (\sigma^{*}B-E).\Gamma=\left(a-2\left[\frac{a}{3}\right]\right) \beta_{2}+\left( b-2\left[\frac{b}{3}\right]\right)\beta_{1}-\gamma > \left[\frac{b}{3}\right]\alpha_{1}-\gamma \geq 0.
\end{equation*}
In the case that $n>0$, then $\beta_{1}>0$, $\beta_{2}\geq \beta_{1} n$. In this case
\begin{equation*}
    (\sigma^{*}B-E).\Gamma=\left( a-2\left[\frac{a}{3}\right]\right) \left(\beta_{2}-\beta_{1}n\right)+\left(b-2\left[\frac{b}{3}\right]\right)\beta_{1}-\gamma \geq \left(b-2\left[\frac{b}{3}\right]\right) \beta_{1}-\gamma > 0.
\end{equation*}
since $\left[\frac{b}{3}\right] \geq 2$ by $(\ref{condicionb})$.\\
\end{proof}
Therefore, we have the following theorem concerning the case of a single node using the conditions of Lemma \ref{lema1nodoveryample} and Theorem \ref{h1L=0}. 
\begin{theorem}\label{teoremasobregaussian1nodo}
    Suppose that $C$ is a $1-$nodal curve on $\mathbb{F}_{n}$, linearly equivalent to $aC_{0}+bF$. Let $\sigma:X=Bl_{p}(\mathbb{F}_{n})\rightarrow\mathbb{F}_{n}$ be the blow-up of $\mathbb{F}_{n}$ at $p$, where $p\in C$ is the node. Let $\widetilde{C}$ be the normalization of $C$ under $\sigma$. Assume that $b\geq an+3$ and $a\geq 6$,
    then the Gaussian map
    \begin{equation*}
        \Phi_{K_{X}+\widetilde{C}}:\bigwedge^{2}H^{0}(X,\mathcal{O}_{X}(K_{X}+\widetilde{C}))\rightarrow H^{0}(X,\Omega^{1}_{X}(2K_{X}+2\widetilde{C}))
    \end{equation*}
    is surjective.
\end{theorem}

\section{Applications and special cases} 
\label{sec:applications and special cases}
\subsection{The 1-nodal case and special cases}
For $\delta=1$, we have the following result
\begin{theorem}\label{teoremafinal1nodo}
    Suppose that $C$ is a $1-$nodal curve on $\mathbb{F}_{n}$, linearly equivalent to $aC_{0}+bF$ with $a,b\geq 0$. Let $\sigma:X=Bl_{p}(\mathbb{F}_{n})\rightarrow\mathbb{F}_{n}$ be the blow-up of $\mathbb{F}_{n}$ at $p$, where $p\in C$ is the node. Let $\widetilde{C}$ be the normalization of $C$ under $\sigma$. Assume that $p\notin C_{0}$, $a\geq 6$ and $b\geq\max\{(a-2)n+6,an+3\}$, then
    \begin{enumerate}
        \item $\Phi_{X,\mathcal{O}_{X}(K_{X}+\widetilde{C})}$ is surjective.\\
        \item $
   \textnormal{corank} \Phi_{\widetilde{C}} = \left\{ \begin{array}{lcc} 9 & & n\leq 2,  \\ &  & \\  n+6 &  & n\geq 3. \end{array} \right.
$
    \end{enumerate}
\end{theorem}
\begin{proof}
    From Theorem \ref{teoremasobregaussian1nodo}, we obtain (1). For (2), combine Remark \ref{remarkcorank2} and Theorem \ref{corangoh0rhofinal}.  
\end{proof}

The commutative diagram in (\ref{diagramaconmuativo1}) of Section \ref{subsec:corankbetaalpha} implies that if $H^{1}(\Omega^{1}_{X}(\log\widetilde{C})(2K_{X}+\widetilde{C}))=0$, then $H^{0}(\rho)$ is surjective. We can find conditions in the coefficients $a$ and $b$ of $aC_{0}+bF$ to get the surjectivity of $H^{0}(\rho)$.
\begin{theorem}
    Suppose $a\geq 5$, and $b\geq (a-2)n+6$, or that $a\geq 5$ and $b\geq 5$ if $n=0$. If $\delta> a+b+\frac{-an}{2}$ such that $g-\delta\geq 2$, then $H^{0}(\rho)$ is surjective.
\end{theorem}
\begin{proof}
    By proposition \ref{propisomlogandh1}, we have that $$H^{1}(X,\Omega^{1}_{X}(\log\widetilde{C})(2K_{X}+\widetilde{C}))\cong H^{1}(\widetilde{C},\mathcal{O}_{\widetilde{C}}(2K_{X}+\widetilde{C})) \cong H^{0}(\widetilde{C},\mathcal{O}_{\widetilde{C}}(-K_{X})).$$ Then $H^{0}(\rho)$ is surjective if $-K_{X}.C=2b+2a-an-2\delta<0$.
\end{proof}
\subsection{Geometric applications}
    \begin{corollary} \label{corolario1nodo} Under the conditions of Theorem \ref{teoremafinaldeltanodos}, if there exists a $\delta-$nodal curve $C\in |aC_{0}+bF|$ in $\mathbb{F}_{n}$, then $C$ cannot be embedded in $\mathbb{F}_{m}$ as $\delta-$nodal curve, for $m\neq n$ and $m\geq 4$.
    \end{corollary}
    \begin{proof}
   The corank of $\Phi_{\widetilde{C}}$ depends on $n$ in a specific way (Theorem \ref{teoremafinaldeltanodos}). If $C$ could be embedded in $\mathbb{F}_{m}$ with $m\neq n$, the corank would have to match the value for $\mathbb{F}_{m}$, creating a contradiction.
    \end{proof}
    
\subsection{Connection to Wahl's Conjecture}
 Our results provide evidence for a general phenomenon:
 \begin{conjecture}\label{conjeturawahl}
     (Wahl \cite{MR1064866}). Let $S$ be a regular surface ($H^{1}(S,\mathcal{O}_{S})=0$). Then there is $g_{0}$ so that for every smooth curve $C$ on $S$ of genus $\geq g_{0}$, one has
     \[
     \textnormal{corank}(\Phi_{C})\geq h^{0}(S,\mathcal{O}_{S}(-K_{S})).
     \]
 \end{conjecture}
Theorem \ref{teoremafinaldeltanodos} verifies this conjecture in our setting. To see this,     recall that $-K_{X}=\sigma^{*}(-K_{\mathbb{F}_{n}})-E$ and $H^{0}(X,\mathcal{O}_{X}(-K_{X}))\cong H^{0}(X,\sigma^{*}\mathcal{O}_{\mathbb{F}_{n}}(-K_{\mathbb{F}_{n}})\otimes\mathcal{O}_{X}(-E))\cong H^{0}(\mathbb{F}_{n}, \mathcal{O}_{\mathbb{F}_{n}}(-K_{\mathbb{F}_{n}})\otimes \mathcal{I}_{Z})$ since, for all $i>0$, $R^{i}\sigma_{*}\mathcal{O}_{X}(-E)=0$ and $\sigma_{*}\mathcal{O}_{X}(-E)=\mathcal{I}_{Z}$ where $\mathcal{I}_{Z}$ is the ideal sheaf of $Z=\{p_{1},\dots,p_{\delta}\}$ on $\mathbb{F}_{n}$. \\
 Consider the exact classical sequence
   \begin{equation*}
       0\rightarrow \mathcal{I}_{Z}\rightarrow\mathcal{O}_{\mathbb{F}_{n}}\rightarrow\bigoplus_{i=1}^{\delta}\mathcal{O}_{p_{i}}\rightarrow 0.
   \end{equation*}
Tensoring this sequence with $\mathcal{O}_{\mathbb{F}_{n}}(-K_{\mathbb{F}_{n}})$ yields:
      \begin{equation*}
       0\rightarrow \mathcal{I}_{z}\otimes\mathcal{O}_{\mathbb{F}_{n}}(-K_{\mathbb{F}_{n}})\rightarrow\mathcal{O}_{\mathbb{F}_{n}}(-K_{\mathbb{F}_{n}})\rightarrow \mathcal{O}_{\mathbb{F}_{n}}(-K_{\mathbb{F}_{n}})\otimes \left (\bigoplus_{i=1}^{\delta}\mathcal{O}_{p_{i}}\right )\rightarrow0
   \end{equation*}
   with long exact sequence in cohomology
      \begin{equation*}
      \begin{split}
           0 & \rightarrow H^{0}(\mathbb{F}_{n},\mathcal{I}_{p}\otimes\mathcal{O}_{\mathbb{F}_{n}}(-K_{\mathbb{F}_{n}}))\rightarrow H^{0}(\mathbb{F}_{n},\mathcal{O}_{\mathbb{F}_{n}}(-K_{\mathbb{F}_{n}})) \xrightarrow{} \mathbb{C}^{\delta}\rightarrow \\
           & \rightarrow H^{1}(\mathbb{F}_{n},\mathcal{I}_{p}\otimes\mathcal{O}_{\mathbb{F}_{n}}(-K_{\mathbb{F}_{n}}))\rightarrow H^{1}(\mathbb{F}_{n},\mathcal{O}_{\mathbb{F}_{n}}(-K_{\mathbb{F}_{n}}))\rightarrow 0 \rightarrow \\
           & \rightarrow H^{2}(\mathbb{F}_{n},\mathcal{I}_{p}\otimes\mathcal{O}_{\mathbb{F}_{n}}(-K_{\mathbb{F}_{n}}))\rightarrow H^{2}(\mathbb{F}_{n},\mathcal{O}_{\mathbb{F}_{n}}(-K_{\mathbb{F}_{n}}))\rightarrow 0. \\
      \end{split}
   \end{equation*}
   Therefore $\textnormal{corank}(\Phi_{\widetilde{C}})=h^{0}(\mathbb{F}_{n},\mathcal{O}_{\mathbb{F}_{n}}(-K_{\mathbb{F}_{n}}))\geq h^{0}(\mathbb{F}_{n},\mathcal{I}_{p}\otimes\mathcal{O}_{\mathbb{F}_{n}}(-K_{\mathbb{F}_{n}}))= h^{0}(X,\mathcal{O}_{X}(-K_{X}))$.

\bibliography{refs}

\providecommand{\bysame}{\leavevmode\hbox to3em{\hrulefill}\thinspace}
\providecommand{\MR}{\relax\ifhmode\unskip\space\fi MR }
\providecommand{\MRhref}[2]{%
  \href{http://www.ams.org/mathscinet-getitem?mr=#1}{#2}
}
\providecommand{\href}[2]{#2}
\begin{thebibliography}{HMPLV23}

\bibitem[ABS17]{MR3710056}
Enrico Arbarello, Andrea Bruno, and Edoardo Sernesi, \emph{On hyperplane
  sections of {K}3 surfaces}, Algebr. Geom. \textbf{4} (2017), no.~5, 562--596.
  \MR{3710056}

\bibitem[BDRS99]{MR1698897}
Mauro~C. Beltrametti, Sandra Di~Rocco, and Andrew~J. Sommese, \emph{On
  generation of jets for vector bundles}, Rev. Mat. Complut. \textbf{12}
  (1999), no.~1, 27--45. \MR{1698897}

\bibitem[BEL91]{MR1092845}
Aaron Bertram, Lawrence Ein, and Robert Lazarsfeld, \emph{Vanishing theorems, a
  theorem of {S}everi, and the equations defining projective varieties}, J.
  Amer. Math. Soc. \textbf{4} (1991), no.~3, 587--602. \MR{1092845}

\bibitem[CHM88]{MR975124}
Ciro Ciliberto, Joe Harris, and Rick Miranda, \emph{On the surjectivity of the
  {W}ahl map}, Duke Math. J. \textbf{57} (1988), no.~3, 829--858. \MR{975124}

\bibitem[CLM00]{MR1760876}
Ciro Ciliberto, Angelo~Felice Lopez, and Rick Miranda, \emph{On the {W}ahl map
  of plane nodal curves}, Complex analysis and algebraic geometry, de Gruyter,
  Berlin, 2000, pp.~155--163. \MR{1760876}

\bibitem[CLS11]{MR2810322}
David~A. Cox, John~B. Little, and Henry~K. Schenck, \emph{Toric varieties},
  Graduate Studies in Mathematics, vol. 124, American Mathematical Society,
  Providence, RI, 2011. \MR{2810322}

\bibitem[CM90]{MR1074304}
Ciro Ciliberto and Rick Miranda, \emph{On the {G}aussian map for canonical
  curves of low genus}, Duke Math. J. \textbf{61} (1990), no.~2, 417--443.
  \MR{1074304}

\bibitem[CU93]{MR1248892}
Fernando Cukierman and Douglas Ulmer, \emph{Curves of genus ten on {$K3$}
  surfaces}, Compositio Math. \textbf{89} (1993), no.~1, 81--90. \MR{1248892}

\bibitem[DM92]{MR1061775}
Jeanne Duflot and Rick Miranda, \emph{The {G}aussian map for rational ruled
  surfaces}, Trans. Amer. Math. Soc. \textbf{330} (1992), no.~1, 447--459.
  \MR{1061775}

\bibitem[EV92]{MR1193913}
H\'el\`ene Esnault and Eckart Viehweg, \emph{Lectures on vanishing theorems},
  DMV Seminar, vol.~20, Birkh\"auser Verlag, Basel, 1992. \MR{1193913}

\bibitem[Ful98]{MR1644323}
William Fulton, \emph{Intersection theory}, second ed., Ergebnisse der
  Mathematik und ihrer Grenzgebiete. 3. Folge. A Series of Modern Surveys in
  Mathematics [Results in Mathematics and Related Areas. 3rd Series. A Series
  of Modern Surveys in Mathematics], vol.~2, Springer-Verlag, Berlin, 1998.
  \MR{1644323}

\bibitem[Har77]{MR463157}
Robin Hartshorne, \emph{Algebraic geometry}, Graduate Texts in Mathematics,
  vol. No. 52, Springer-Verlag, New York-Heidelberg, 1977. \MR{463157}

\bibitem[HMPLV23]{MR4665627}
S.~Huh, S.~Marchesi, J.~Pons-Llopis, and J.~Vall\`es, \emph{Generalized
  logarithmic sheaf on smooth projective surfaces}, Int. Math. Res. Not. IMRN
  (2023), no.~21, 18387--18442. \MR{4665627}

\bibitem[Laz04]{MR2095471}
Robert Lazarsfeld, \emph{Positivity in algebraic geometry. {I}}, Ergebnisse der
  Mathematik und ihrer Grenzgebiete. 3. Folge. A Series of Modern Surveys in
  Mathematics [Results in Mathematics and Related Areas. 3rd Series. A Series
  of Modern Surveys in Mathematics], vol.~48, Springer-Verlag, Berlin, 2004,
  Classical setting: line bundles and linear series. \MR{2095471}

\bibitem[MM83]{MR726433}
Shigefumi Mori and Shigeru Mukai, \emph{The uniruledness of the moduli space of
  curves of genus~{$11$}}, Algebraic geometry ({T}okyo/{K}yoto, 1982), Lecture
  Notes in Math., vol. 1016, Springer, Berlin, 1983, pp.~334--353. \MR{726433}

\bibitem[Muk88]{MR977768}
Shigeru Mukai, \emph{Curves, {$K3$} surfaces and {F}ano {$3$}-folds of genus
  {$\leq 10$}}, Algebraic geometry and commutative algebra, {V}ol.\ {I},
  Kinokuniya, Tokyo, 1988, pp.~357--377. \MR{977768}

\bibitem[Rei88]{MR932299}
Igor Reider, \emph{Vector bundles of rank {$2$} and linear systems on algebraic
  surfaces}, Ann. of Math. (2) \textbf{127} (1988), no.~2, 309--316.
  \MR{932299}

\bibitem[Ser06]{MR2247603}
Edoardo Sernesi, \emph{Deformations of algebraic schemes}, Grundlehren der
  mathematischen Wissenschaften [Fundamental Principles of Mathematical
  Sciences], vol. 334, Springer-Verlag, Berlin, 2006. \MR{2247603}

\bibitem[Ser18]{MR3832409}
\bysame, \emph{The {W}ahl map of one-nodal curves on {K}3 surfaces}, Local and
  global methods in algebraic geometry, Contemp. Math., vol. 712, Amer. Math.
  Soc., [Providence], RI, [2018] \copyright 2018, pp.~307--315. \MR{3832409}

\bibitem[Wah90]{MR1064866}
Jonathan Wahl, \emph{Gaussian maps on algebraic curves}, J. Differential Geom.
  \textbf{32} (1990), no.~1, 77--98. \MR{1064866}

\bibitem[Wah92]{MR1201392}
\bysame, \emph{Introduction to {G}aussian maps on an algebraic curve}, Complex
  projective geometry ({T}rieste, 1989/{B}ergen, 1989), London Math. Soc.
  Lecture Note Ser., vol. 179, Cambridge Univ. Press, Cambridge, 1992,
  pp.~304--323. \MR{1201392}

\end{thebibliography}

\noindent
\textsc{Centro de Ciencias Matemáticas, UNAM, Campus Morelia, Morelia, Michoacán, México.}

\noindent
\textit{Email address}: mguerrero@matmor.unam.mx
\end{document}